\documentclass[11pt,reqno]{amsart}
\usepackage[utf8]{inputenc}
\usepackage{amsmath,amssymb}
\usepackage{hyperref}
\usepackage{mathtools}
\usepackage{cite}
\usepackage{tikz}
\usetikzlibrary{calc}
\usetikzlibrary{automata}
\usepackage{amsmath}
\usepackage{mathrsfs}
\usepackage{polyglossia}
\setdefaultlanguage{english}
\usepackage{fontspec}
\usepackage{enumitem}
\usepackage{float}
\numberwithin{equation}{section}
\newtheorem{thm}{Theorem}[section]

\newtheorem{lem}[thm]{Lemma}

\theoremstyle{remark}

\theoremstyle{example}

\newcommand{\abs}[1]{\left\vert#1\right\vert}
\theoremstyle{definition}
\newcommand{\bb}[1]{\textcolor{blue}{{#1}}}
\newcommand{\cc}[1]{\textcolor{red}{{#1}}}
\newcommand{\yy}[1]{\textcolor{cyan}{{#1}}}
\newcommand{\malcev}{\mathbin{\hbox{$\bigcirc$\rlap{\kern-8.25pt\raise0,50pt\hbox{${\tt
  m}$}}}}}
\newcommand{\smalcev}{\mathbin{\hbox{$\bigcirc$\rlap{\kern-7pt\raise0,30pt\hbox{${\tt
  m}$}}}}}

\newcommand{\Aug}{\mathop{\mathrm{Aug}}\nolimits}
\newcommand{\J}{\mathrel{\mathscr J}} % J - relation
\newcommand{\R}{\mathrel{\mathscr R}} % R - relation
\newcommand{\eL}{\mathrel{\mathscr L}} % L - relation
\newcommand{\HH}{\mathrel{\mathscr H}}
\newcommand{\K}{\mathrel{\mathscr K}}
\newcommand{\rk}{\mathop{\mathrm{rk}}\nolimits}

\begin{document}
\title[Simplicity of Augmentation Submodules of Rank More than Two]
{Simplicity of Augmentation Submodules in Monoids with 0-Minimal Ideals of Rank Greater than Two}
\author{M. H. Shahzamanian}
\address{M. H. Shahzamanian\\ Centro de Matem\'atica e Departamento de Matem\'atica, Faculdade de Ci\^{e}ncias,
Universidade do Porto, Rua do Campo Alegre, 687, 4169-007 Porto,
Portugal}
\email{m.h.shahzamanian@gmail.com}

\begin{abstract}

In this paper, we construct explicit families of transformation monoids whose augmentation submodules are simple and whose associated 0-minimal $\J$-classes have rank greater than two.  
These examples provide new monoids with simple augmentation submodules and non-complete associated graphs.  
We also establish a connection between the sandwich matrix of a 0-minimal $\J$-class of rank two and the simplicity of the corresponding augmentation module, yielding a criterion that determines simplicity directly from the rank of this matrix for this class of monoids.
\end{abstract}
\maketitle
%\tableofcontents

\section{Introduction}

The representation theory of finite monoids has attracted substantial attention in last decades, driven by interactions with probability, combinatorics, and theoretical computer science. A key stimulus was the work of Bidigare, Hanlon, and Rockmore, who introduced methods from monoid representation theory into the study of finite-state Markov chains~\cite{BHR}. Their approach was further developed by Brown and Diaconis~\cite{DiaconisBrown1} and subsequently expanded in a wide range of works
\cite{Brown1,Brown2,bjorner2,GrahamChung,Saliolaeigen,sandpile,AyyerKleeSchilling,ayyer_schilling_steinberg_thiery.2013,randomwalksrings}. An overview of these developments can be found in~\cite[Chapter~14]{Ben-Rep-Monoids-2016}.%, while connections with fast Fourier transforms and data analysis are discussed in
%\cite{Malandro1,Malandro2,Malandro3}.

Monoid representation theory has also found significant applications in the study of finite-dimensional algebras arising from discrete geometry, including those associated with hyperplane arrangements, oriented matroids, and CAT(0) cube complexes. In particular, several works have explored deep connections between representations of finite monoids and the representation theory of finite-dimensional algebras and quivers; see, for example,
\cite{Putcharep3,DO,Saliola,rrbg,globaltn,itamar1, Ben-Sha2,Sha-Det,Sha-Det2,Ste-Fac-det}.

The representation theory of finite monoids provides a powerful framework for studying finite monoids acting on finite sets; see, for example,
\cite{Ben-Transformation-Monoids-2010,Ben-Rep-Monoids-2016}. 
Associated with any transformation monoid is a transformation module, which plays a role analogous to the permutation module of a group. Such modules arise naturally in applications to Markov chains and automata theory.
In automata theory, every finite-state automaton determines a transformation monoid, and representation-theoretic properties of the associated transformation module can often be used to investigate structural and algorithmic questions; see, for instance,
\cite{Perrincomplred,berstelperrinreutenauer}. 
This perspective has been particularly influential in work on the \v{C}ern\'y conjecture~\cite{cerny,VolkovLata, Volkov2, Volkov3}, an over 60-year-old problem in automata theory. In mathematical terms, the conjecture asserts that if $A$ is a set of mappings on an $n$-element set such that the monoid generated by $A$ contains a constant mapping, then there exists a product of at most $(n-1)^2$ elements of $A$ (with repetitions allowed) that is a constant map. In many of the deepest results on this problem, the augmentation submodule plays a central role.

More precisely, let $(M,\Omega)$ be a finite transformation monoid and $\mathbb{F}$ a field. 
Then $\mathbb{F}\Omega$ is the associated transformation module, where the action of $M$ on $\Omega$ is extended linearly.  
The augmentation submodule $\Aug(\mathbb{F}\Omega)$ consists of all formal linear combinations of elements of $\Omega$ whose coefficients sum to zero.
It is a classical result, going back to Burnside, that if $G$ is a transitive permutation group on $\Omega$, then $\Aug(\mathbb{C}\Omega)$ is simple if and only if $G$ is $2$-transitive.  
Similarly, $(G,\Omega)$ is $2$-homogeneous, i.e., acts transitively on unordered pairs of elements of $\Omega$, if and only if $\Aug(\mathbb{R}\Omega)$ is simple~\cite{Ar-Ca-St,cameron}.  
Motivated by the results of~\cite{synchgroups}, Steinberg asked John Dixon when $\Aug(\mathbb{Q}\Omega)$ is simple.  
Dixon made partial progress in~\cite{Dix}, showing that such permutation groups are primitive of either affine type or almost simple, and classifying the affine type examples.  
He also showed that if $\Aug(\mathbb{Q}\Omega)$ is simple, then $G$ is $3/2$-transitive.  
The classification of $3/2$-transitive groups was later obtained in~\cite{Bam}, and in particular, the classification of permutation groups with simple $\Aug(\mathbb{Q}\Omega)$ was completed in~\cite[Corollary~1.6]{Bam}.

In~\cite{Aug-2020}, the authors characterized when the augmentation module of a transformation monoid is simple over a field, assuming that the corresponding question for permutation groups is already known, as is the case for the fields $\mathbb{C}$, $\mathbb{R}$, and $\mathbb{Q}$.  
They showed that not all $2$-transitive transformation monoids have simple augmentation modules over $\mathbb{C}$; in fact, an additional condition is required: the incidence matrix of a certain associated set system must have full rank.  
They also demonstrated that $2$-transitivity is not necessary for simplicity.
For a transformation monoid $M$, the simplicity of its augmentation module is determined by five conditions, one of which involves an associated graph $\Gamma(M)$.  
Moreover, in all cases where the augmentation module is simple, $M$ has a $0$-minimal $\J$-class.

In~\cite{Aug-2020}, several examples of monoids with simple augmentation submodules of rank two for their $0$-minimal $\J$-class were given.  
However, examples of monoids whose $0$-minimal $\J$-class has rank greater than two and whose associated graph $\Gamma(M)$ is not complete were not provided.
In this paper, we construct examples of monoids whose augmentation submodules are simple of arbitrary rank greater than two for their $0$-minimal $\J$-class over the complex numbers $\mathbb{C}$, and whose associated graph $\Gamma(M)$ is not complete.  
Moreover, for a transformation monoid $(M,\Omega)$ with $0$-minimal $\J$-class $J$, we establish a connection between the sandwich matrix of $J$ and the simplicity of the augmentation submodules of $M$ when the rank of $J$ is equal to two.  
As a consequence, the simplicity of the augmentation module can be determined directly by computing the rank of the sandwich matrix of $J$.

The paper is organized as follows. We begin by recalling background on monoids, their representation theory, and the simplicity of the augmentation modules of transformation monoids.
Then, in Section~\ref{SoASR2}, we present our criterion for the simplicity of the augmentation module of a transformation monoid with a $0$-minimal $\J$-class of rank two, based on the rank of the sandwich matrix of this $\J$-class.
In Section~\ref{SASRG2}, we construct transformation monoids with simple augmentation modules whose $0$-minimal $\J$-class has rank greater than two and whose associated graph is not complete.

%%%%%%%%%%%%%%%%%%%%%%%%%%%%%%%%%%%%%%%%%%%%%%%%%%%%%%%%%%%%%%%%%%%%%%%%%%%%%%%%%%%%%%%%%%%%%%%%%%%%%%%%%%%%%%%%%%%%%%%%%%%%%%%%%%%%%%%%%%%%%%%%%%%%%%%%%%%%%%%%%%%%%%%%%%%%%%%%%%%%%%%%%%%%%%%%%%%%%%%%%%%%%%%%%%%%%%%%%%%%%%%%%%%%%%%%%%%%%%%%%%%%%%%%%%%%%%

\section{Preliminaries}
\subsection{Monoids}
For standard notation and terminology relating to monoids, we refer the reader to~\cite{Alm,Cli-Pre,Rho-Ste}.

Let $M$ a finite monoid. Let $a,b\in M$. We say that $a\R b$ if $aM = bM$, $a\eL b$ if $Ma = Mb$ and $a\HH b$ if $a\R b$ and $a\eL b$. Also, we say that $a\J b$, if $MaM = MbM$.
The relations $\R,\eL$, $\HH$ and $\J$ are Green relations and all of them are equivalence relations were first introduced by Green~\cite{Gre}.
For an element \(a \in M\), we denote by \(R_a\), \(L_a\), \(H_a\), and \(J_a\) the $\R$, $\eL$, $\HH$ and $\J$-classes containing \(a\), respectively.

An important property of finite monoids is the stability property that $J_m\cap Mm = L_m$ and $J_m\cap mM = R_m$, for every $m \in M$.
For $\J$-classes $J_a$ and $J_b$, we can define the partial order $\leq$ as follows:
$$MaM\subseteq MbM\mbox{ if and only if }J_a\leq J_b.$$

An element $e$ of $M$ is called idempotent if $e^2 = e$. The set of all idempotents of $M$ is denoted by $E(M)$; more generally, for any $X\subseteq M$, we put $E(X)=X\cap E(M)$.
An idempotent $e$ of $M$ is the identity of the monoid $eMe$. The group of units $G_e$ of $eMe$ is called the maximal subgroup of $M$ at $e$. Note that $G_e=H_e$.
We denote the minimal ideal of $M$ by $I(M)$.

An element $m$ of $M$ is called (von Neumann) regular if there exists an element $n\in M$ such that $mnm=m$. Note that an element $m$ is regular if and only if $m\eL e$, for some $e\in E(M)$, if and only if $m\R f$, for some $f\in E(M)$. A $\J$-class $J$ is regular if all its elements are regular, if and only if $J$ has an idempotent, if and only if $J^2\cap J\neq\emptyset$.  Note that if $N$ is a submonoid of $M$ and $a,b\in N$ are regular in $N$, then $a\K b$ in $N$ if and only if $a\K b$ in $M$
where $\K$ is any of $\R, \eL$ or $\HH$ (\cite[Proposition A.1.16]{Rho-Ste}).

Let $G$ be a group, $n$ and $m$ be integers and $P$ be an $m\times n$ matrix with entries in $G\cup\{0\}$.
The Rees matrix semigroup $\mathcal{M}^{0}(G, n,m;P)$ is the set of all triples $(i,g,j)$ where $g\in G$, $1\leq i \leq n$ and $1\leq j\leq m$, together with  $0$, and the following binary operation between non-zero elements
\begin{equation*}
(i,g,j)(i',g',j')  = \begin{cases}
  (i,gp_{ji'}g',j')& \text{if}\ p_{ji'}\neq 0;\\
  0& \text{otherwise},
\end{cases}
\end{equation*}
for every $(i,g,j),(i',g',j')\in \mathcal{M}^{0}(G, n,m;P)$ where $P=(p_{ij})$. The Rees matrix semigroup $\mathcal{M}^{0}(G, n,m;P)$ is regular if and only if each row and each column of $P$ contains a non-zero entry, in which case all non-zero elements are $\J$-equivalent.

\subsection{$M$-sets}
An $M$-set, for a monoid $M$, consists of a set $\Omega$ together with a mapping $M \times \Omega \rightarrow \Omega$, written $(m, \omega) \mapsto m\omega$ and
called an action, such that:
\begin{enumerate}
\item $1\omega = \omega$;
\item $m_2(m_1\omega) = (m_2m_1)\omega$,
\end{enumerate}
for every $\omega\in\Omega$ and $m_1,m_2\in M$.
The pair $(M,\Omega)$ is called a transformation monoid if $M$ acts faithfully on the $\Omega$. We write $T_{\Omega}$ for the full transformation monoid on $\Omega$, that is, the monoid of all self-maps of $\Omega$. Transformation monoids on $\Omega$ amount
to submonoids of $T_{\Omega}$.
We say that $M$ is transitive on $\Omega$ if $M\omega = \Omega$ for all $\omega\in\Omega$. The rank of $m \in M$ is defined by $\rk(m) = \abs{m\Omega}$.
The rank is constant on $\J$-classes and, for a transformation monoid, the minimum rank is attained precisely on the minimal ideal $I(M)$.

%A non-empty subset $\Delta$ of $\Omega$ is $M$-invariant if $M\Delta \subseteq \Delta$.
%The set $\Omega^2$ is a finite transformation monoid via $m(\alpha, \beta) = (m\alpha, m\beta)$, for every $(\alpha, \beta)\in\Omega^2$. 
Let $\Delta = \{(\alpha, \alpha) \in \Omega^2 \mid \alpha \in \Omega\}$. One says that $M$ acts $2$-transitively on $\Omega$, if for every $(\alpha,\beta),(\alpha',\beta')\in \Omega^2\setminus \Delta$ there exists an element $m\in M$ such that $m(\alpha,\beta)=(\alpha',\beta')$.
We also say that $M$ is $0$-transitive on $\Omega$, if for a necessarily unique $\omega_0\in \Omega$, $M\omega_0 = \{\omega_0\}$ and $M\omega =\Omega$, for all $\omega\in\Omega\setminus\{\omega_0\}$. %Traditionally, one uses $0$ instead of $\omega_0$ for the fixed point of $M$ but to avoid confusion with the $0$ of associated vector spaces, we write $\omega_0$. But we still use the term "$0$-transitive".

%A congruence on an $M$-set $\Omega$ is an equivalence relation $\equiv$ such that $\alpha\equiv\beta$ implies $m\alpha\equiv m\beta$ for all $\alpha,\beta \in\Omega$ and $m \in M$. In this case, the quotient $\Omega/{\equiv}$ becomes an $M$-set in the natural way and the quotient map $\Omega\rightarrow \Omega/{\equiv}$ is a morphism. A transformation monoid $(M,\Omega)$ is primitive if it admits no non-trivial proper congruences.
%We refer the reader for more details on this concept to~\cite{Ben-Transformation-Monoids-2010}.

\subsection{Transformation modules and representations of monoids}
Let $(M,\Omega)$ be a finite transformation monoid and $\mathbb{F}$ a field. By extending the action of $M$ on $\Omega$ linearly, as the basis, $\mathbb{F}\Omega$ is a left $\mathbb{F}M$-module where $\mathbb{F}M$ is the monoid algebra of $M$ on $\mathbb{F}$. It is the transformation module associated with the action. Also, we have that
$\mathbb{F}^\Omega = \{f\colon  \Omega \rightarrow \mathbb{F}\}$ is a right $\mathbb{F}M$-module by putting $(fm)(\omega) = f(m\omega)$.
We identify $\mathbb{F}^\Omega$ as the dual space to $\mathbb{F}\Omega$ in the natural way.

The map $\eta\colon  \mathbb{F}\Omega \rightarrow \mathbb{F}$ sending each element of $\Omega$ to $1$ is called the augmentation map and hence we write $\ker \eta = \Aug(\mathbb{F}\Omega)$ which is the augmentation submodule of $\mathbb{F}\Omega$. 

Let $\mathcal{W}$ be an $\mathbb{F}M$-submodule of $\mathbb{F}^\Omega$. The $\mathbb{F}M$-submodule
$\mathcal{W}^{\bot}$ of $\mathbb{F}\Omega$ is the null space of $\mathcal{W}$ as follows:
$$\{v \in \mathbb{F}\Omega\mid w(v) = 0,\mbox{ for every }w\in\mathcal{W}\}.$$

We recall the map $1_B\in \mathbb{F}^\Omega$, for a subset $B\in \Omega$, defined as follows:
\begin{equation*}
1_B(x) = \begin{cases}
  1& \text{if}\ x\in B;\\
  0& \text{otherwise},
\end{cases}
\end{equation*}
for every $x\in\Omega$.

\subsection{Simplicity of Augmentation Submodules in Monoids}\label{Sim-Auq}
Let \(n>1\) be a positive integer,  let \( \Omega = \{\omega_1, \ldots, \omega_n\} \), and $\mathbb{F}$ be a field.  
Suppose that \( (M, \Omega) \) is a finite transformation monoid.  
In \cite{Aug-2020}, the simplicity of augmentation submodules of \( (M, \Omega) \) is characterized in the case where \( M \) is not a group.  
It is shown that if \( M \leq T_{\Omega} \) is not a group and the augmentation submodule \( \Aug(\mathbb{F}\Omega) \) is simple over a field \( \mathbb{F} \), then \( I(M) \) contains a constant map, and \( M \setminus I(M) \) has a unique minimal \( \J \)-class \( J \), which is regular.

%Let \( E(J) \) denote the set of idempotents in the \( \J \)-class \( J \).  
We define a graph \( \Gamma(M) = (\Omega, E) \), where  
\[
E = \left\{ \{\omega_i, \omega_j\} \,\middle|\, i\neq j, e\omega_i = \omega_i \text{ and } e\omega_j = \omega_j, \text{ for some } e \in E(J) \right\}.
\]

Next, define the set system
\[
\mathcal{E} = \left\{ B \,\middle|\, B = f^{-1}f\omega \text{ for some } f \in E(J) \text{ and } \omega \in \Omega \right\},
\]
which is a collection of subsets of \( \Omega \).
The incidence matrix of \( \mathcal{E} \), denoted \( I(\mathcal{E}) \), is the \( \Omega \times \mathcal{E} \) matrix defined by
\[
I(\mathcal{E})_{\omega,B} =
\begin{cases}
    1, & \text{if } \omega \in B, \\[4pt]
    0, & \text{otherwise},
\end{cases}
\qquad \text{for all } \omega \in \Omega \text{ and } B \in \mathcal{E}.
\]

By~\cite[Theorem 3.4]{Aug-2020}, the augmentation submodule \( \Aug(\mathbb{F}\Omega) \) is simple if and only if the following conditions hold:
\begin{enumerate}
    \item the monoid $M$ contains a constant map;
\item the subset $M\setminus I(M)$ has a unique minimal $\J$-class $J$ and moreover $J$ is regular;
\item if $e\in E(J)$, then $\Aug(\mathbb{F}e\Omega)$ is a simple $\mathbb{F}G_e$-module;
\item the rank of the incidence matrix of the set system 
\[
\{\, B \mid B = f^{-1}f\omega \text{ for some } f \in E(J) \text{ and } \omega \in \Omega \,\}
\]
is $\abs{\Omega}$ over $\mathbb{F}$;
\item the graph $\Gamma(M)$ is connected.
\end{enumerate}

Let
\[
\mathcal{W}
    = \left\langle\, 1_{B}
        \;\middle|\;
        B = f^{-1}f\omega 
        \text{ for some } f \in E(J) \text{ and } \omega \in \Omega 
      \right\rangle_{\mathbb{F}}.
\]
If condition~(4) holds, then \(\mathcal{W} = \mathbb{F}^{\Omega}\), and may be replaced by the requirement that
\(\mathcal{W}^{\perp} = 0\), which is the form we use in this paper.

When \( \Aug(\mathbb{F}\Omega) \) is simple, it follows from the above theorem that the rank of the \( \J \)-class \( J \), viewed as the rank of its mappings, is greater than one.
In this case, we say that \( \Aug(\mathbb{F}\Omega) \) is of the rank of the mappings of \( J \), which is an integer greater than one.
This integer will be referred to as the \emph{rank} of the augmentation module.

%%%%%%%%%%%%%%%%%%%%%%%%%%%%%%%%%%%%%%%%%%%%%%%%%%%%%%%%%%%%%%%%%%%%%%%%%%%%%%%%%%%%%%%%%%%%%%%%%%%%%%%%%%%%%%%%%%%%%%%%%%%%%%%%%%%%%%%%%%%%%%%%%%%%%%%%%%%%%%%%%%%%%%%%%%%%%%%%%%%%%%%%%%%%%%%%%%%%%%%%%%%%%%%%%%%%%%%%%%%%%%%%%%%%%%%%%%%%%%%%%%%%%%%%%%%%%%%%%%%%%%%%%%%%%%%%%%%%%%%%%%%%%%%%%%%%%%%%%%%%%%%%%%%%%%%%%%%%%%%%%%%%%%%%%%%%%%%%%%%%%%%%%%%%%%%%%%%%%%%%%%%%%%%%%%%%%%%%%%%%%%%%%%%%%%%%%%%%%%%%%%%%%%%%%%%%%%%%%%%%%%%%%%

\section{Simplicity of Augmentation Submodules of Rank 2}\label{SoASR2}

Throughout this section, we assume that \( \mathbb{F} \) is a field, and that \( (M, \Omega) \) satisfies the following conditions:
\begin{enumerate}
    \item \( n > 1 \);
    \item \( I(M) \) contains a constant map;
    \item \( M \setminus I(M) \) has a unique minimal \( \J \)-class \( J \), which is regular;
    \item the rank of the elements of \( J \) is equal to 2.
\end{enumerate}

In this section, we establish a connection between the Rees sandwich matrix of the $\J$-class $J$ and the simplicity of the augmentation submodules of $(M,\Omega)$, which culminates in a theorem at the end of the section.

Let $\Gamma$ be a graph with vertex set $\Omega = \{\omega_1, \dots, \omega_n\}$.  
We define the \emph{difference set} of $\Gamma$ by
\[
\Delta(\Gamma) = \left\{ \omega_i - \omega_j \;\middle|\; 1 \le i < j \le n \text{ and } \{\omega_i,\omega_j\} \in E(\Gamma) \right\}.
\]

We also define the incidence matrix ${I(\mathcal{E})}^{\Gamma}$ indexed by
$\Delta(\Gamma) \times \mathcal{E}$ as follows:
\[
{I(\mathcal{E})}^{\Gamma}_{\omega_i - \omega_j,\, B}
= I(\mathcal{E})_{\omega_i,\, B} - I(\mathcal{E})_{\omega_j,\, B},
\]
where $\omega_i - \omega_j \in \Delta(\Gamma)$ and $B \in \mathcal{E}$.

\begin{lem}\label{IEGamma}
Suppose that the graph \( \Gamma \) is a tree on \( \Omega \).  
Then the incidence matrix $I(\mathcal{E})$ has rank $n$ over $\mathbb{F}$ if and only if
the matrix ${I(\mathcal{E})}^{\Gamma}$ has rank $n-1$ over $\mathbb{F}$.
\end{lem}

\begin{proof}
Let $R_1, \ldots, R_n$ denote the rows of the incidence matrix $I(\mathcal{E})$.

First, suppose that $\operatorname{rank}(I(\mathcal{E})) = n$.  
Assume that
\begin{equation}\label{RiRj}
\sum_{\omega_i - \omega_j \in \Delta(\Gamma)}
\alpha_{\omega_i - \omega_j}\,(R_i - R_j) = 0,
\end{equation}
for some scalars $\alpha_{\omega_i - \omega_j} \in \mathbb{F}$.

Let
\[
A_i
= \sum_{\omega_i - \omega \in \Delta(\Gamma)} \alpha_{\omega_i - \omega}
  - \sum_{\omega - \omega_i \in \Delta(\Gamma)} \alpha_{\omega - \omega_i},
\]
for each $1 \le i \le n$. Since $\operatorname{rank}(I(\mathcal{E})) = n$, it follows from~\eqref{RiRj} that
\[
\sum_{i=1}^n A_i R_i = 0.
\]
Hence, we must have $A_i = 0$ for every $1 \le i \le n$.

Let $\omega_i$ be a vertex of degree one.  
Then there exists a vertex $\omega_j$ such that either
$\omega_i - \omega_j \in \Delta(\Gamma)$ or $\omega_j - \omega_i \in \Delta(\Gamma)$.
By symmetry, we may assume that $\omega_i - \omega_j \in \Delta(\Gamma)$.
Then we have $A_i = \alpha_{\omega_i - \omega_j}$, and since $A_i = 0$,
it follows that $\alpha_{\omega_i - \omega_j} = 0$.
Hence, removing vertices of degree one from the graph \( \Gamma \) does not affect the equation in~\eqref{RiRj}.  
Proceeding recursively, we can remove all such vertices one by one without changing the relation.  
It follows that \( \alpha_{\omega_i - \omega_j} = 0 \) for every \( \omega_i - \omega_j \in \Delta(\Gamma) \).  

Therefore, the rank of the matrix \( {I(\mathcal{E})}^{\Gamma} \) is \( n - 1 \), the number of the edges of the tree graph $\Gamma$.

Now, suppose that the rank of the matrix ${I(\mathcal{E})}^{\Gamma}$ is $n-1$ over $\mathbb{F}$. 
To prove that $\operatorname{rank}(I(\mathcal{E})) = n$, we construct a basis of $I(\mathcal{E})$ consisting of $n$ linearly independent rows.

Let \( f \in E(J) \), and suppose that \( f\Omega = \{\omega_{k_1}, \ldots, \omega_{k_m}\} \).  
Since \( \Gamma \) is a tree, we have
\[
R_k \in \left\langle R_{k_1},\, R_i - R_j \mid \omega_i - \omega_j \in \Delta(\Gamma) \right\rangle_{\mathbb{F}}
\]
for every \( 1 \leq k \leq n \). Thus, our chosen basis is
\[
R_{k_1}, \quad R_i - R_j \;\; \text{for all } \omega_i - \omega_j \in \Delta(\Gamma).
\]

We prove that the row $R_{k_1}$ cannot be generated by  elements $R_i - R_j$ where $ \omega_i - \omega_j \in \Delta(\Gamma)$.  
Suppose, to the contrary, that
\begin{equation}\label{RR1}
R_{k_1} = \sum_{\omega_i - \omega_j \in \Delta(\Gamma)} 
\beta_{\omega_i - \omega_j}\,(R_i - R_j),
\end{equation}
for some scalars $\beta_{\omega_i - \omega_j} \in \mathbb{F}$.
Let \( B_l = f^{-1}\omega_{k_l} \), for every $1\leq l\leq m$.  
Since  
\[
I(\mathcal{E})_{\omega_i, B_{k_1}} + \sum_{l \in \{k_2, \ldots, k_m\}} I(\mathcal{E})_{\omega_i, B_l} = 1, \quad \text{for all } 1 \leq i \leq n,
\]
it follows that for every \( \omega_i - \omega_j \in \Delta(\Gamma) \), we have either
\[
{I(\mathcal{E})}^{\Gamma}_{\omega_i - \omega_j, B_{k_1}} =  \sum_{l \in \{k_2, \ldots, k_m\}} {I(\mathcal{E})}^{\Gamma}_{\omega_i - \omega_j, B_l}=0,
\]
or
\[
{I(\mathcal{E})}^{\Gamma}_{\omega_i - \omega_j, B_{k_1}} = -\sum_{l \in \{k_2, \ldots, k_m\}} {I(\mathcal{E})}^{\Gamma}_{\omega_i - \omega_j, B_l} .
\]

Now, by~\eqref{RR1}, we have
\begin{align*}
I(\mathcal{E})_{\omega_{k_1}, B_{k_1}} 
&= \sum_{\omega_i - \omega_j \in \Delta(\omega)} \beta_{\omega_i - \omega_j} \, {I(\mathcal{E})}^{\Gamma}_{\omega_i - \omega_j, B_{k_1}} \\
&= -\sum_{\omega_i - \omega_j \in \Delta(\omega)} \beta_{\omega_i - \omega_j} \sum_{l \in \{k_2, \ldots, k_m\}} {I(\mathcal{E})}^{\Gamma}_{\omega_i - \omega_j, B_l} \\
&= -\sum_{l \in \{k_2, \ldots, k_m\}} \sum_{\omega_i - \omega_j \in \Delta} \beta_{\omega_i - \omega_j} \, {I(\mathcal{E})}^{\Gamma}_{\omega_i - \omega_j, B_l} \\
&= -\sum_{l \in \{k_2, \ldots, k_m\}} I(\mathcal{E})_{\omega_{k_1}, B_l}.
\end{align*}

Thus, we obtain
\[
I(\mathcal{E})_{\omega_{k_1}, B_{k_1}} + \sum_{l \in \{k_2, \ldots, k_m\}} I(\mathcal{E})_{\omega_{k_1}, B_l} = 0,
\]
which contradicts the fact that this sum equals 1.  
Therefore, the rows \( R_{k_1} \) and \( R_i - R_j \) for \( \omega_i - \omega_j \in \Delta(\omega) \) are linearly independent.  
It follows that the rank of the incidence matrix \( I(\mathcal{E}) \) is \( n \) over \( \mathbb{F} \).
\end{proof}

\begin{lem}\label{IEGammaNonCon}
Suppose that the graph $\Gamma$ is not connected and has $\abs{\Omega}-1$ edges.
The rank of the matrix ${I(\mathcal{E})}^{\Gamma}$ is less than $n-1$ over $\mathbb{F}$.
\end{lem}

\begin{proof}
Since the graph $\Gamma$ is not connected and $\Gamma$ has $\abs{\Omega}-1$ edges, there is a cycle path $i_1,\ldots, i_m$ with length $m\leq\abs{\Omega}-1$ in $\Gamma$. Then, the rows $i_1-i_2,i_2-i_3,\ldots,i_{m-1}-i_m,i_1-i_m$ are not linearly independent in the matrix ${I(\mathcal{E})}^{\Gamma}$. So, the rank of the matrix ${I(\mathcal{E})}^{\Gamma}$ is less than $n-1$ over $\mathbb{F}$.
\end{proof}

We denote by \( J^0 \) the semigroup \( J \cup \{0\} \), where \( 0 \notin J \). The multiplication in \( J^0 \) is defined by:
\[
j_1 \cdot j_2 = 
\begin{cases}
j_1 j_2 & \text{if } j_1 j_2 \in J, \\
0 & \text{otherwise}.
\end{cases}
\]

%Suppose that the rank of the elements in \( J \) is equal to 2. 
Define the following sets:
\[
N_1 = \left\{ B_k \,\middle|\, B_k = e_k^{-1}e_k\omega_1,\ \text{for some } 1 \leq k \leq \abs{E(J)} \right\},
\]
and
\[
N_2 = \left\{ (\omega_i, \omega_j) \,\middle|\, 
\begin{aligned}
&\text{there exists an idempotent } e \in E(J) \text{ such that} \\
&e\Omega = \{\omega_i, \omega_j\} \text{ and } i < j
\end{aligned}
\right\}.
\]

Since the rank of the elements in \( J \) is equal to 2, the maximal subgroups of \( J \) are either trivial or isomorphic to a group of order 2.
Now, as \( J \) is regular, the semigroup \( J^0 \) is isomorphic to the Rees matrix semigroup
\[
M' = \mathcal{M}^{0}(G, N_1, N_2; P),
\]
where \( G \) is either the trivial group or the group of order 2; namely \( G = \{1_G, g\} \). 
The set $N_1$ consists of subsets $B_k$ that form the domain of $J$.  
For each $B_k \in N_1$, the corresponding elements of $J$ are $B_k$ and its complement $\Omega \setminus B_k$.
The set \( N_2 \) represents the image sets. 

The sandwich matrix \( P = (p_{(\omega_i,\omega_j),k}) \) is an \( N_2 \times N_1 \) matrix defined as follows: for each \( 1 \leq k \leq \abs{E(J)} \) and each \( (\omega_i, \omega_j) \in N_2 \),
\[
p_{(\omega_i,\omega_j),k} =
\begin{cases}
0 & \text{if } \omega_i, \omega_j \in B_k \text{ or } \omega_i, \omega_j \notin B_k, \\
1_G & \text{if } \omega_i \in B_k \text{ and } \omega_j \notin B_k, \\
g & \text{if } \omega_i \notin B_k \text{ and } \omega_j \in B_k.
\end{cases}
\]
%In the case where the group \( G \) is trivial, we assume \( g = 1_G \).

We define the \( N_2 \times N_1 \) matrix \( P' = (p'_{(\omega_i,\omega_j),k}) \) from the matrix \( P \) as follows: for each \( (\omega_i, \omega_j) \in N_2 \) and \( 1 \leq k \leq \abs{E(J)} \),
\[
p'_{(\omega_i,\omega_j),k} =
\begin{cases}
0 & \text{if } p_{(\omega_i,\omega_j),k} = 0, \\
1 & \text{if } p_{(\omega_i,\omega_j),k} = 1_g, \\
-1 & \text{if } p_{(\omega_i,\omega_j),k} = g.
\end{cases}
\]

Equivalently, the edge set of the graph $\Gamma(M)$ can be written as
\[
E(\Gamma(M))
= \left\{ \{\omega_i, \omega_j\} \;\middle|\; \exists\, e \in E(J)
\text{ such that } e\Omega = \{\omega_i, \omega_j\} \right\}.
\]
Then it is easily verified that $P' = {I(\mathcal{E})}^{\Gamma(M)}$.

The following theorem gives a criterion for the simplicity of the augmentation submodule
$\Aug(\mathbb{F}\Omega)$ of $(M,\Omega)$ in terms of the rank of the matrix $P'$.

\begin{thm}\label{MainThm}
The augmentation submodule $\Aug(\mathbb{F}\Omega)$ is simple if and only if the rank of the matrix $P'$ is equal to $n-1$.
\end{thm}

\begin{proof}
First, suppose that the rank of the matrix $P'$ is equal to $n-1$. 
Then the rank of the matrix ${I(\mathcal{E})}^{\Gamma(M)}$ is equal to $n-1$. 
%Then there exists a graph $\Gamma$ with $n-1$ edges such that $\Gamma$ is a subgraph of $\Gamma(M)$ and the rank of the matrix ${I(\mathcal{E})}^{\Omega}$ is equal to $n-1$.
Hence, by Lemma~\ref{IEGammaNonCon}, there exists a tree $\Gamma$ with $n-1$ edges
such that $\Gamma$ is a subgraph of $\Gamma(M)$ and
\[
\operatorname{rank}\big({I(\mathcal{E})}^{\Gamma}\big) = n-1.
\]
Thus, by Lemma~\ref{IEGamma}, the incidence matrix $I(\mathcal{E})$ has rank $n$
over $\mathbb{F}$. Moreover, the existence of such a graph $\Gamma$ implies that
$\Gamma(M)$ is connected. Therefore, the augmentation submodule
$\Aug(\mathbb{F}\Omega)$ is simple.

Now, suppose that the rank of the matrix $P'$ is less than $n-1$.
If the graph $\Gamma(M)$ is not connected, then the augmentation submodule
$\Aug(\mathbb{F}\Omega)$ is not simple.  
If $\Gamma(M)$ is connected, then for every tree subgraph $\Gamma$ of $\Gamma(M)$,
the rank of the matrix ${I(\mathcal{E})}^{\Gamma}$ is less than $n-1$.
Thus, by Lemma~\ref{IEGamma}, the incidence matrix $I(\mathcal{E})$ has rank less than $n$.
Therefore, the augmentation submodule $\Aug(\mathbb{F}\Omega)$ is not simple.
\end{proof}

%%%%%%%%%%%%%%%%%%%%%%%%%%%%%%%%%%%%%%%%%%%%%%%%%%%%%%%%%%%%%%%%%%%%%%%%%%%%%%%%%%%%%%%%%%%%%%%%%%%%%%%%%%%%%%%%%%%%%%%%%%%%%%%%%%%%%%%%%%%%%%%%%%%%%%%%%%%%%%%%%%%%%%%%%%%%%%%%%%%%%%%%%%%%%%%%%%%%%%%%%%%%%%%%%%%%%%%%%%%%%%%%%%%%%%%%%%%%%%%%%%%%%%%%%%%%%%

\section{Simplicity of Augmentation Submodules of Rank greater than two}\label{SASRG2}

Let \((M,\Omega)\) be as in Section~\ref{Sim-Auq}.
By~\cite[Theorem~3.4]{Aug-2020}, if \(\Aug(\mathbb{F}\Omega)\) is simple,
then \(I(M)\) contains a constant map and
\(M \setminus I(M)\) has a unique minimal \(\mathcal{J}\)-class \(J\),
which is regular.

In~\cite{Aug-2020}, several examples of monoids with simple augmentation
submodules of rank two are given.
It is also observed that adjoining the constant maps to certain
highly transitive monoids, such as groups of  invertible affine mappings
 or symmetric groups, yields examples of higher rank.
In all such cases, the associated graph \(\Gamma(M)\) is complete.
More generally, every \(2\)-transitive transformation monoid has a
complete graph \(\Gamma(M)\)
(see~\cite[Proposition~3.6]{Aug-2020}).

In this section, we construct examples of monoids whose augmentation
submodules are simple of arbitrary rank greater than two over the
complex numbers \(\mathbb{C}\), and whose associated graph
\(\Gamma(M)\) is not complete.
These examples illustrate genuinely more intricate situations than
those previously considered.

We note that if \(\Aug(\mathbb{F}\Omega)\) is simple, then the monoid
\(M\) is either transitive or \(0\)-transitive on \(\Omega\)
\cite[Lemma~3.1]{Aug-2020}.
In the \(0\)-transitive case, the rank of \(\Aug(\mathbb{F}\Omega)\) is
always two \cite[Theorem~5.2]{Aug-2020}.
Hence, in the present work, the examples we present are such that
\(\Aug(\mathbb{F}\Omega)\) is simple and \(M\) acts transitively on \(\Omega\).

Every mapping \(m \in M\) can be described by two tuples of size \(r_m\):
\[
(B_{i_1}, \ldots, B_{i_{r_m}}) \quad \text{and} \quad
(\omega_{i_1}, \ldots, \omega_{i_{r_m}}),
\]
where \(r_m\) is the rank of \(m\), \(\{B_{i_1}, \ldots, B_{i_{r_m}}\}\) is a
partition of \(\Omega\), and
\[
m(B_{i_j}) = \omega_{i_j} \quad \text{for each } 1 \le j \le r_m,
\]
meaning that \(m\) maps every element of \(B_{i_j}\) to \(\omega_{i_j}\).

With the above description of elements of \(M\) in terms of their kernel
partitions and image tuples, the \(\R\)- and
\(\eL\)-classes of \(J\) can be described as follows.
Define
\[
\mathcal{B} = \bigl\{ \{B_{i_1}, \ldots, B_{i_r}\} \mid 1 \le i \le n_1 \bigr\}
\quad \text{and} \quad
\mathcal{I} = \bigl\{ \{\omega_{j_1}, \ldots, \omega_{j_r}\} \mid 1 \le j \le n_2 \bigr\},
\]
where \(r\) is the rank of \(J\), and \(n_1\) and \(n_2\) denote the numbers
of \(\R\)- and \(\eL\)-classes of \(J\), respectively.
Two elements of \(J\) belong to the same \( \R \)-class if their first tuples, viewed as sets from \(\mathcal{B}\), coincide; they belong to the same \( \eL \)-class if their second tuples, viewed as sets from \(\mathcal{I}\), coincide.
Note that two elements may lie in the same \( \R \)-class even if the entries of their tuples in \(\mathcal{B}\) appear in a different order, since these tuples represent the same partition of \(\Omega\). 
Similarly, two elements may lie in the same \( \eL \)-class even if their tuples in \(\mathcal{I}\) differ by a reordering, as they represent the same image set of size \(r\).

Suppose that \(\mathcal{B}\) is a set of partitions of \(\Omega\) into \(r\) blocks, and
that \(\mathcal{I}\) is a set of subsets of \(\Omega\), each of cardinality \(r\).
It is natural to expect that, for such sets \(\mathcal{B}\) and \(\mathcal{I}\), one can construct a transformation monoid that contains a constant map
and whose minimal \(\J\)-class is described by these two sets.
However, for arbitrary choices of these sets, there need not exist a
\(\J\)-class \(J\) of \(M \setminus I(M)\) determined by them, let alone one
that is the unique minimal \(\J\)-class.
In fact, if these sets satisfy our requirements, then for every 
\( 1 \leq i \leq n_{1} \) and \( 1 \leq j \leq n_{2} \), one of the following conditions must hold:
\begin{enumerate}[label=(\roman*), ref=(\roman*)]
    \item\label{cond:i} there exists \( 1 \leq k \leq r \) such that \(\{\omega_{j_1}, \ldots, \omega_{j_r} \} \subseteq B_{i_k};\) 
    \item\label{cond:ii} there exists a permutation \( \sigma \in S_r \) such that \(\omega_{j_k} \in B_{i_{\sigma(k)}}\) for all \(1 \leq k \leq r\).
\end{enumerate}
Otherwise, \( J \) is not the minimal \( \J \)-class, and there exists another 
\( \J \)-class strictly between \( J \) and the class of constant maps. 
When the sets \(\mathcal{B}\) and \(\mathcal{I}\) satisfy the above conditions, we say that they
satisfy the \emph{\(\J\)-minimal compatibility condition}.

Of course, for \( J \) to be regular, such a permutation must exist for every row of \( \mathcal{B} \) 
and every column of \( \mathcal{I} \).  
Moreover, if only the second condition holds for all pairs \( (i,j) \), 
then the constant maps must be added to \( J \).

In this way, we can construct the Rees matrix semigroup
\begin{equation}\label{ConsM}
\mathcal{M}^{0}(S_{r},\, n_{1},\, n_{2};\, P),
\end{equation}
where the matrix \( P = (p_{j,i}) \) is defined as follows.  
For every \( 1 \leq i \leq n_{1} \) and \( 1 \leq j \leq n_{2} \), we set
\[
p_{ji} = 0 
\quad \text{if and only if} \quad
\{ \omega_{j_1}, \ldots, \omega_{j_r} \} \subseteq B_{i_k}
\ \text{for some } k,
\]
and otherwise
\[
p_{j,i} \in S_{r}
\quad \text{is the permutation satisfying} \quad
\omega_{j_k} \in B_{i_{\,p_{j,i}(k)}}
\quad \text{for all } 1 \leq k \leq r.
\]
Then the monoid \( M \) is constructed as the union of the nonzero elements of the Rees matrix semigroup
\( \mathcal{M}^{0}(S_{r},\, n_{1},\, n_{2};\, P) \), viewed as transformations, together with the identity on
\( \Omega \) and all constant maps.  
(We note that, in general, for a monoid whose augmentation submodule
\(\Aug(\mathbb{F}\Omega)\) is simple, the maximal subgroup of the
\( \J \)-class \( J \) need not be exactly \( S_{r} \); it may be any subgroup of \( S_{r} \).
However, for the purposes of our construction, and to simplify the verification of the simplicity conditions,
we assume that the maximal subgroup of \( J \) is \( S_{r} \) itself.)

Thus, our plan for constructing such examples is to specify collections
\( \mathcal{B} \) and \( \mathcal{I} \), as described above, in order to obtain a monoid $M$.
As explained above, the collections \( \mathcal{B} \) and \( \mathcal{I} \) must satisfy the \(\J\)-minimal compatibility condition, 
as well as the requirements ensuring that \( J \) is regular and that constant maps behave appropriately.
Finally, we verify the remaining simplicity conditions (3)–(5) from
\cite[Theorem~3.4]{Aug-2020}. 

In fact, from the set \( \mathcal{I} \) we may directly construct the graph \( \Gamma(M) \) and verify its connectivity.  
Let \( c \in \mathcal{I} \). Since we assume that \( J \) is regular, every row of the matrix \( P \) contains a nonzero entry.  
Hence, there exists a column \( r \) such that \( p_{c r} \neq 0 \), and therefore  
\((p_{c r}^{-1}, r, c)\) is an idempotent of \( J \).  
This shows that all elements in the row indexed by \( c \) lie in a complete subgraph of \( \Gamma(M) \).
Moreover, every idempotent of \( J \) has the form described above and connects all elements in its corresponding row.  
Thus, in order to verify the connectivity of the graph \( \Gamma(M) \), it suffices to check that these complete subgraphs (one for each subset of \( \mathcal{I} \)) are connected to one another through the vertices lying in their intersections.  
If the elements of \( \mathcal{I} \) can be partitioned into two or more collections such that no element of one collection intersects any element of the others, then we conclude that \( \Gamma(M) \) is not connected.

Moreover, since the symmetric group $S_r$ is 2-transitive on the set $\{1,\ldots,r\}$,
the augmented module $\Aug(\mathbb{C}e\Omega)$ is a simple $\mathbb{C}G_e$-module, for every idempotent $e$ in $M$ for which $G_e\cong S_r$.
%Proposition B.12[Ben, Page 284]. Let $G$ be a finite group and $\Omega$ a transitive $G$-set. Then
%$G$ is 2-transitive if and only if $Aug(\mathbb{C}\omega)$ is simple.
Thus, the main difficulty lies in the verification of condition (4). 
We also note that the duality condition in (4) can be checked as mentioned in Section~\ref{Sim-Auq}, namely,
\(\mathcal{W}^{\perp} = 0\).

The following lemma yields a useful bound on the cardinality of \(\mathcal{B}\); it expresses \(|\mathcal{B}|\) in terms of the rank of the minimal \(\J\)-class \(J\) and the size \(|\Omega|\), under the assumption that the augmentation submodule \(\Aug(\mathbb{C} \Omega)\) is simple.

\begin{lem}\label{lem:boundB}
If the augmentation submodule \(\Aug(\mathbb{C}\Omega)\) is simple, then
\begin{equation*}
|\Omega| \;=\; r \quad \text{if } |\mathcal{B}| = 1, \qquad
|\Omega| \;<\; r\,|\mathcal{B}| \quad \text{if } |\mathcal{B}| > 1,
\end{equation*}
where \(r\) is the rank of the minimal \(\J\)-class \(J\).
\end{lem}

\begin{proof}
If \(|\Omega| = 1\), the lemma holds trivially.%we have $r=|\mathcal{B}|=1$ and thus $|\Omega| = r|\mathcal{B}|=1$.

Then, we suppose that \(|\Omega| > 1\).  

First, assume that \(|\mathcal{B}| = 1\). If \(|\Omega| > r\), then the incidence matrix of the associated set system over \(\mathbb{C}\) nevertheless has rank \( |\Omega| \), since its rows are repeated for some elements of $\Omega$. Otherwise that rank of matrix is equal to $r$.

Now, assume that \(|\mathcal{B}| > 1\).
In the incidence matrix of this set system, each row contains exactly \(|\mathcal{B}|\) entries equal to \(1\), with \(|\Omega|\) of these ones distributed across the \(r\) columns corresponding to each partition.  
If we multiply the first \(r\) columns, corresponding to the classes of the
first partition, by \(-( |\mathcal{B}|-1 )\) and add them to the remaining columns, the sum of all columns becomes zero.  
This shows that the column rank is strictly less than \(r |\mathcal{B}|\), and therefore the row rank is also strictly less than \(r |\mathcal{B}|\).  
Consequently, we obtain the desired inequality.
\end{proof}

In the next subsection, we construct examples of monoids whose augmentation
submodules are simple of rank three over the complex numbers \(\mathbb{C}\),
and whose associated graph \(\Gamma(M)\) is not complete.  
Subsequently, we construct examples of monoids whose augmentation submodules are simple of rank greater than three, first treating the case of rank four separately, and then considering odd and even ranks greater than four in two distinct subsections.

%%%%%%%%%%%%%%%%%%%%%%%%%%%%%%%%%%%%%%%%%%%%%%%%%%%%%%%%%%%%%%%%%%%%%%%%%%%%%%%%%%%%%%%%%%%%%%%%%%%%
%%%%%%%%%%%%%%%%%%%%%%%%%%%%%%%%%%%%%%%%%%%%%%%%%%%%%%%%%%%%%%%%%%%%%%%%%%%%%%%%%%%%%%%%%%%%%%%%%%%%
%%%%%%%%%%%%%%%%%%%%%%%%%%%%%%%%%%%%%%%%%%%%%%%%%%%%%%%%%%%%%%%%%%%%%%%%%%%%%%%%%%%%%%%%%%%%%%%%%%%%

\subsection{($r=3$)}\label{r=3}
We choose the set 
\[
\mathcal{I} = \big\{ \{\omega_1, \omega_2, \omega_3\}, \{\omega_3, \omega_4, \omega_5\}, \{\omega_5, \omega_6, \omega_1\} \big\}
\]
where \(\Omega = \{\omega_1, \ldots, \omega_6\}\).

\begin{center}\begin{figure}[h]
\begin{tikzpicture}[scale=0.9, every node/.style={font=\small}]
  % Corner coordinates of a smaller equilateral triangle
 % \coordinate (A) at (94:2);    % top
  \coordinate (A) at (86:2);    % top
  \coordinate (B) at (214:2);   % bottom-left
%  \coordinate (BB) at (206:2);   % bottom-left
 % \coordinate (C) at (334:2);   % bottom-right
  \coordinate (C) at (326:2);   % bottom-right

  % Midpoints on each side
  \coordinate (mABB) at ($(A)!0.5!(B)$);
  \coordinate (mBCC) at ($(B)!0.5!(C)$);
  \coordinate (mCAA) at ($(C)!0.5!(A)$);

  % Draw triangle edges
  \draw[line width=0.9pt] (A) -- (B);
  \draw[line width=0.9pt] (B) -- (C);
  \draw[line width=0.9pt] (C) -- (A);

  % Draw and label corner vertices
  \node[draw=white, inner sep=1.5pt, above=3pt] at (A) {$\omega_3$};
  \node[draw=white, inner sep=1.5pt, left=3pt] at (B) {$\omega_1$};
  \node[draw=white, inner sep=1.5pt, right=3pt] at (C) {$\omega_5$};

  % Draw mid-edge vertices
  \node[draw=white, inner sep=1.5pt, left=3pt] at (mABB) {$\omega_2$};
  \node[draw=white, inner sep=1.5pt, below=3pt] at (mBCC) {$\omega_6$};
  \node[draw=white, inner sep=1.5pt, right=3pt] at (mCAA) {$\omega_4$};
  \fill (A) circle (2pt);
 % \fill (AA) circle (2pt);
\fill (B) circle (2pt);
%\fill (BB) circle (2pt);
\fill (C) circle (2pt);
%\fill (CC) circle (2pt);

\fill (mABB) circle (2pt);
\fill (mBCC) circle (2pt);
\fill (mCAA) circle (2pt);

\end{tikzpicture}
\caption{Triangular cycle of subsets of \(\Omega\)}
\label{fig:n=3-1}\end{figure}
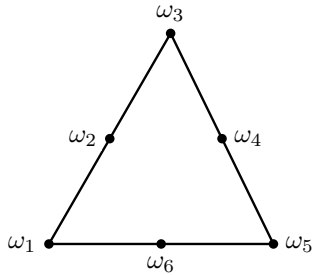
\end{center}
We illustrate the set \( \mathcal{I} \) in Figure~\ref{fig:n=3-1}, where each subset of \( \mathcal{I} \) is represented by a line connecting its elements.  
To simplify the presentation, we ignore the fact that each subset of \( \mathcal{I} \) forms a complete subgraph in the graph \( \Gamma(M) \), and we depict it only as a single line.  
Whenever a vertex belongs to the intersection of two subsets, it is placed at the joint of the two lines in the figure.
However, even with this simplified method of illustrating the set \( \mathcal{I} \), which we use throughout this paper, one can easily reconstruct the graph \( \Gamma(M) \).

Now, we wish to construct a monoid \( M \) using a suitable collection \( \mathcal{B} \) such that our set \( \mathcal{I} \) is as defined above.  
Since \( \mathcal{B} \) and \(\mathcal{I}\) must satisfy the \(\J\)-minimal compatibility condition, we will show that the subsets in \(\mathcal{B} \) must be chosen from the following:
\begin{enumerate}[label=(\arabic*), ref=(\arabic*)]
\item\label{cond:1} $\{\{\omega_1,\omega_2,\omega_3\},\{\omega_4,\omega_6\},\{\omega_5\}\}$;
\item\label{cond:2} $\{\{\omega_3,\omega_4,\omega_5\},\{\omega_2,\omega_6\},\{\omega_1\}\}$;
\item\label{cond:3} $\{\{\omega_5,\omega_6,\omega_1\},\{\omega_2,\omega_4\},\{\omega_3\}\}$;
\item\label{cond:4} $\{\{\omega_2,\omega_5\},\{\omega_4,\omega_1\},\{\omega_6,\omega_3\}\}$.
\end{enumerate}

Suppose that \( B_i \) is a partition of \( \Omega \) consisting of three parts.  
To prove that \( B_i \) must be one of the cases listed above, we consider the following two possibilities:
\begin{enumerate}
\item a subset of \( \mathcal{I} \) is entirely contained in a part of the partition \( B_i \);
\item no subset of \( \mathcal{I} \) is entirely contained in a part of \( B_i \).
\end{enumerate}

First, suppose that the first case holds.  
Namely, assume that  
\[
\{\omega_1,\omega_2,\omega_3\} \subseteq B_{i_k}
\]
for some part \( B_{i_k} \) of \( B_i \).  

If \( \omega_4 \) also lies in \( B_{i_k} \), then the \(\J\)-minimal compatibility condition imply that \( \omega_5 \in B_{i_k} \) and, consequently, \( \omega_6 \in B_{i_k} \).  
In that case, \( B_{i_k} = \Omega \), which is a contradiction.  
The same argument applies if \( \omega_5 \) or \( \omega_6 \) were added to \( B_{i_k} \).  
Therefore, we must have
\[
B_{i_k} = \{\omega_1,\omega_2,\omega_3\}.
\]
Moreover, the elements \( \omega_4 \) and \( \omega_5 \) cannot lie in the same part of \( B_i \); similarly, \( \omega_6 \) and \( \omega_5 \) cannot lie in the same part of \( B_i \).  
Hence, \( \omega_5 \) must form a singleton subset of \( B_i \), which forces the remaining elements \( \omega_4 \) and \( \omega_6 \) to lie together in the remaining part of \( B_i \).  
In this case, item~\ref{cond:1} holds.

By the same argument, if either  
\[
\{\omega_3,\omega_4,\omega_5\} \quad \text{or} \quad \{\omega_5,\omega_6,\omega_1\}
\]  
is contained in a part of \( B_i \), we obtain the partitions listed in items~\ref{cond:2} and~\ref{cond:3}, respectively.
Thus, in this case, the only possibilities for \( B_i \) are precisely those in items~\ref{cond:1}, \ref{cond:2}, and~\ref{cond:3}.

Now, suppose that no subset of \( \mathcal{I} \) is entirely contained in a part of \( B_i \).  
Assume that \( \omega_2 \in B_{i_k} \) for some part \( B_{i_k} \) of \( B_i \). Then \( \omega_1, \omega_3 \notin B_{i_k} \), so there exist distinct parts \( B_{i_{k'}} \) and \( B_{i_{k''}} \) containing \( \omega_1 \) and \( \omega_3 \), respectively. Clearly, \( k, k', k'' \) are pairwise distinct.  
Moreover, the assumption implies that \( \omega_5 \notin B_{i_{k'}} , B_{i_{k''}} \).  
If \( \omega_4 \in B_{i_k} \), then \( \omega_5 \notin B_{i_k} \), leaving no part to place \( \omega_5 \), a contradiction.  
Hence, \( \omega_4 \notin B_{i_k} \), and similarly, \( \omega_6 \notin B_{i_k} \).  
Since \( \omega_5 \notin B_{i_{k'}} \) and \( \omega_5 \notin B_{i_{k''}} \), it follows that \( \omega_5 \in B_{i_k} \). 

Similarly, \( \omega_4 \in B_{i_{k'}} \) and \( \omega_6 \in B_{i_{k''}} \).  
Thus, item~\ref{cond:4} holds.

By Lemma~\ref{lem:boundB},
the set \( \mathcal{B} \) cannot have cardinality one or two.
Then, we face three possibilities for constructing the set \( \mathcal{B} \) as follows:
\begin{enumerate}
    \item \( \mathcal{B} \) contains all of items~\ref{cond:1}, \ref{cond:2}, \ref{cond:3}, and~\ref{cond:4}.
    \item \( \mathcal{B} \) contains only items~\ref{cond:1}, \ref{cond:2}, and~\ref{cond:3}.
    \item \( \mathcal{B} \) contains all items except one of~\ref{cond:1}, \ref{cond:2}, or~\ref{cond:3}.
\end{enumerate}

One can verify that the rank of the incidence matrix of the set system \( \mathcal{B} \) over \( \mathbb{C} \) is 6 in every case.  
Therefore, the following theorem holds.

%BB2:=[
%[[1,2,3],[4,6],[5]],
%[[3,4,5],[2,6],[1]],
%[[5,6,1],[2,4],[3]],
%[[2,5],[4,1],[6,3]]];
%
%CheckRankForSetB := function(n, m, BB)
%    local chosenPartitions, B, i, j, C, allParts, mm;
%
%    # all choices of m partitions from BB
%    C := Combinations(BB, m);
%
%    for chosenPartitions in C do
%
%        allParts := Concatenation( List(chosenPartitions, x -> x) );
%		
%		mm := Length(allParts);
%		B := List( [1..n], i -> List( [1..mm], j -> 0 ) );
%		# Fill B: B[i][j] = 1 if i in allParts[j]
%		for j in [1..mm] do
%		for i in allParts[j] do
%			B[i][j] := 1;
%		od;
%		od;
%
%        # Check rank
%        if RankMat(B) <> n then
%            Print("Matrix B does NOT have full rank for chosen partitions:\n");
%            Print(chosenPartitions, "\n");
%            Display(B);
%			Print("\n\n");
%        fi;
%    od;
%end;
%
%CheckRankForSetB(6, 3, BB2);
%CheckRankForSetB(6, 4, BB2);

\begin{thm}\label{thm:n=3}
Let \( \mathcal{I} \) be the set depicted in Figure~\ref{fig:n=3-1}. Then for every subset \( \mathcal{B} \) together $\mathcal{I}$ satisfying the \(\J\)-minimal compatibility condition with \( |\mathcal{B}| \in \{3,4\} \), the monoid \( M \) constructed in~\eqref{ConsM} possesses a simple augmentation submodule.
\end{thm}

%%%%%%%%%%%%%%%%%%%%%%%%%%%%%%%%%%%%%%%%%%%%%%%%%%%%%%%%%%%%%%%%%%%%%%%%%%%%%%%%%%%%%%%%%%%%%%%%%%%%%%%%%%%%%%%%
%%%%%%%%%%%%%%%%%%%%%%%%%%%%%%%%%%%%%%%%%%%%%%%%%%%%%%%%%%%%%%%%%%%%%%%%%%%%%%%%%%%%%%%%%%%%%%%%%%%%%%%%%%%%%%%%

\subsection{$(r=4)$}\label{r=4}
We extend the construction from the case \(r=3\) to \(r=4\) by choosing the set
\begin{align*}
\mathcal{I} = \big\{& 
\{\omega_1, \omega_2, \omega_3, \omega_4\}, 
\{\omega_4, \omega_5, \omega_6, \omega_7\}, 
\{\omega_7, \omega_8, \omega_9, \omega_{10}\},
\{\omega_{10}, \omega_{11}, \omega_{12}, \omega_1\},\\
&\{\omega_2, \omega_5, \omega_8, \omega_{11}\}
 \big\}
\end{align*}
where \(\Omega = \{\omega_1, \ldots, \omega_{12}\}\).

As in Figure~\ref{fig:n=3-1}, we illustrate the set \(\mathcal{I}\) again in
Figure~\ref{fig:n=4-2}.  
Since the subset \(\{\omega_2, \omega_5, \omega_8, \omega_{11}\}\) cannot be
represented by a single line in the figure, we distinguish this subset by
using a blue color.
\begin{center}
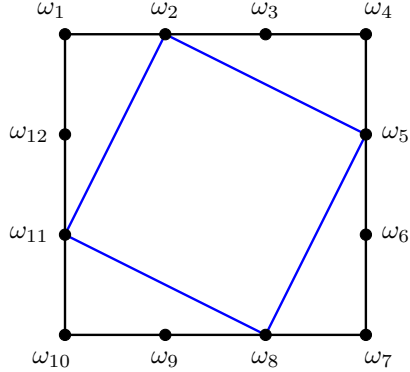
\begin{figure}[h]
\centering
\begin{tikzpicture}[scale=2, every node/.style={font=\small}]
  % Corner coordinates of a square
  \coordinate (A) at (0,2);    % top-left
  \coordinate (B) at (2,2);    % top-right 
  \coordinate (C) at (2,0);    % bottom-right  
  \coordinate (D) at (0,0);    % bottom-left

  % Mid-edge vertices (two per edge)
  \coordinate (AB1) at ($(A)!1/3!(B)$);
  \coordinate (AB2) at ($(A)!2/3!(B)$);

  \coordinate (BC1) at ($(B)!1/3!(C)$);
  \coordinate (BC2) at ($(B)!2/3!(C)$);

  \coordinate (CD1) at ($(C)!1/3!(D)$);
  \coordinate (CD2) at ($(C)!2/3!(D)$);

  \coordinate (DA1) at ($(D)!1/3!(A)$);
  \coordinate (DA2) at ($(D)!2/3!(A)$);

  % Draw square edges
  \draw[line width=0.9pt] (A) -- (B);
  \draw[line width=0.9pt] (B) -- (C);
  \draw[line width=0.9pt] (C) -- (D);
  \draw[line width=0.9pt] (D) -- (A);
  \draw[line width=0.9pt, blue] (AB1) -- (BC1);\draw[line width=0.9pt, blue] (BC1) -- (CD1);
  \draw[line width=0.9pt, blue] (CD1) -- (DA1);\draw[line width=0.9pt, blue] (DA1) -- (AB1);

  % Draw and label corner vertices
  \foreach \v/\l in {A/$\omega_1\ \ \ $,AB1/$\omega_2$, AB2/$\omega_3$,B/$\ \ \ \omega_4$}
    \node[draw=black, circle, inner sep=1.5pt, fill=black, label=above:\l] at (\v) {};
   \foreach \v/\l in {BC1/$\omega_5$, BC2/$\omega_6$}
    \node[draw=black, circle, inner sep=1.5pt, fill=black, label=right:\l] at (\v) {}; 
   \foreach \v/\l in {C/$\ \ \ \omega_7$,CD1/$\omega_8$,CD2/$\omega_9$,D/$\omega_{10}\ \ \ $}
   \node[draw=black, circle, inner sep=1.5pt, fill=black, label=below:\l] at (\v) {};
  % Draw mid-edge vertices
  \foreach \v/\l in { DA1/$\omega_{11}$, DA2/$\omega_{12}$}
    \node[draw=black, circle, inner sep=1.5pt, fill=black, label=left:\l] at (\v) {};
    
\end{tikzpicture}
\caption{Square of subsets of \(\Omega\)}
\label{fig:n=4-2}
\end{figure}
\end{center}

Suppose that \( B_i \) is a partition of \( \Omega \) into four parts that satisfies the
\(\J\)-minimal compatibility condition.

We consider the following two possibilities:
\begin{enumerate}
\item a subset of \( \mathcal{I} \) is entirely contained in one part of the partition \( B_i \);
\item no subset of \( \mathcal{I} \) is entirely contained in any part of \( B_i \).
\end{enumerate}

Assume that the first case holds, and suppose that 
\(\{\omega_1,\ldots,\omega_4\}\) is contained in one part of the partition \(B_i\), 
which we denote by \( B_{i_k} \). 
If \( \omega_5 \in B_{i_k} \), then since \( B_i \) satisfies the
\(\J\)-minimal compatibility condition and \( \omega_2 \in B_{i_k} \), 
we obtain
\[
\{\omega_2,\omega_5,\omega_8,\omega_{11}\} \subseteq B_{i_k}.
\]
Furthermore, this implies that
\[
\{\omega_4,\omega_5,\omega_6,\omega_7\} \subseteq B_{i_k}
\quad\text{and}\quad
\{\omega_{10},\omega_{11},\omega_{12},\omega_1\} \subseteq B_{i_k}.
\]
Consequently, we also have
\[
\{\omega_7,\omega_8,\omega_9,\omega_{10}\} \subseteq B_{i_k}.
\]
Therefore \( B_{i_k} = \Omega \), a contradiction.  
We also conclude that 
\(\omega_8, \omega_{11} \notin B_{i_k}\).
Similarly, we have 
\(\omega_6,\omega_7, \omega_{10}, \omega_{12} \notin B_{i_k}\);  
otherwise, it would imply that 
\(\omega_5 \in B_{i_k}\) or \(\omega_{11} \in B_{i_k}\), 
which is impossible.

Moreover, since the elements \(\omega_5, \omega_8,\) and \(\omega_{11}\) cannot lie in \(B_{i_k}\), and \(B_{i_k}\) contains \(\omega_2\), it follows that each of the elements \(\omega_5, \omega_8,\) and \(\omega_{11}\) must lie in a different part of the partition \(B_i\).  
Thus they must belong to three pairwise distinct parts of \(B_i\), which we denote by 
\(B_{i_{k_1}}, B_{i_{k_2}}, B_{i_{k_3}}\), respectively.

Thus we are left with two possibilities for the part \( B_{i_{k_2}} \):  
either  
\[
\{\omega_7, \omega_8, \omega_9, \omega_{10}\}\subseteq B_{i_{k_2}},
\]
or  
\[
\omega_7, \omega_9, \omega_{10}\not\in B_{i_{k_2}}.
\]

If \(\{\omega_7, \omega_8, \omega_9, \omega_{10}\} \subseteq B_{i_{k_2}}\), 
then we obtain \(\omega_6 \in B_{i_{k_3}}\) and \(\omega_{12} \in B_{i_{k_1}}\).
Therefore, the partition \(B_i\) must be
\begin{equation}\label{BB1}
\big\{
\{\omega_1, \omega_2, \omega_3, \omega_4\},
\{\omega_7, \omega_8, \omega_9, \omega_{10}\},
\{\omega_5, \omega_{12}\},
\{\omega_{11}, \omega_6\}
\big\}.
\end{equation}

Otherwise, we have \(\omega_7, \omega_9, \omega_{10} \notin B_{i_{k_2}}\). 
Hence, we must have \(\omega_7 \in B_{i_{k_3}}\), \(\omega_{10} \in B_{i_{k_1}}\), and \(\omega_9 \in B_{i_k}\). 
Since \(\omega_7 \in B_{i_{k_3}}\), it follows that \(\omega_6 \notin B_{i_{k_3}}\), and thus \(\omega_6 \in B_{i_{k_2}}\). 
Similarly, \(\omega_{12} \in B_{i_{k_2}}\). 
Therefore, the partition \(B_i\) is
\begin{equation}\label{BB2}
\big\{
\{\omega_1, \omega_2, \omega_3, \omega_4, \omega_9\},
\{\omega_5, \omega_{10}\},
\{\omega_7, \omega_{11}\},
\{\omega_6, \omega_8, \omega_{12}\}
\big\}.
\end{equation}

For the cases 
\(\{\omega_4, \omega_5, \omega_6, \omega_7\}\), 
\(\{\omega_7, \omega_8, \omega_9, \omega_{10}\}\), 
and 
\(\{\omega_{10}, \omega_{11}, \omega_{12}, \omega_1\}\), 
when any one of them is contained in a part of \( B_i \), 
we obtain the same conclusion as above.  
In this situation, the only remaining case to check is whether
\[
\{\omega_2, \omega_5, \omega_8, \omega_{11}\} 
\subseteq B_{i_k}
\]
for some part \( B_{i_k} \) of \( B_i \).
If any element of 
\(\Omega \setminus \{\omega_2, \omega_5, \omega_8, \omega_{11}\}\)
lies in \( B_{i_k} \), then the \(\J\)-minimal compatibility condition
imply that \( B_{i_k} = \Omega \), a contradiction.  
Hence we must have 
\[
B_{i_k} = \{\omega_2, \omega_5, \omega_8, \omega_{11}\}.
\]

The elements \( \omega_1 \), \( \omega_3 \), and \( \omega_4 \) then lie in pairwise distinct parts of \( B_i \), 
which we denote by \( B_{i_{k_1}}, B_{i_{k_2}}, B_{i_{k_3}} \), respectively.
We have two possibilities: either 
\(\omega_6 \in B_{i_{k_1}}\) and \(\omega_7 \in B_{i_{k_2}}\), or 
\(\omega_7 \in B_{i_{k_1}}\) and \(\omega_6 \in B_{i_{k_2}}\).

The first case yields
\begin{equation}\label{BB3}
B_i =
\big\{
\{\omega_2, \omega_5, \omega_8, \omega_{11}\},\;
\{\omega_1, \omega_6, \omega_9\},\;
\{\omega_3, \omega_7, \omega_{12}\},\;
\{\omega_4, \omega_{10}\}
\big\}.
\end{equation}

The second case gives two possibilities for \( B_i \):
\begin{equation}\label{BB4}
\big\{
\{\omega_2, \omega_5, \omega_8, \omega_{11}\},\;
\{\omega_1, \omega_7\},\;
\{\omega_3, \omega_6, \omega_{10}\},\;
\{\omega_4, \omega_9, \omega_{12}\}
\big\},
\end{equation}
or
\begin{equation}\label{BB5}
\big\{
\{\omega_2, \omega_5, \omega_8, \omega_{11}\},\;
\{\omega_1, \omega_7\},\;
\{\omega_3, \omega_6, \omega_9, \omega_{12}\},\;
\{\omega_4, \omega_{10}\}
\big\}.
\end{equation}

Now, we consider the remaining case, where no subset of \( \mathcal{I} \) is entirely contained in any part of \( B_i \). 
For this case, the following possibilities arise for \( B_i \):
\begin{equation}\label{BB6}
\big\{
\{\omega_2, \omega_6, \omega_{10}\},\;
\{\omega_5, \omega_1, \omega_9\},\;
\{\omega_8, \omega_4, \omega_{12}\},\;
\{\omega_{11}, \omega_3, \omega_7\}
\big\},
\end{equation}
or
\begin{equation}\label{BB7}
\big\{
\{\omega_2, \omega_7, \omega_{12}\},\;
\{\omega_5, \omega_3, \omega_{10}\},\;
\{\omega_8, \omega_1, \omega_6\},\;
\{\omega_{11}, \omega_4, \omega_9\}
\big\}.
\end{equation}

Hence \(B\) must consist of the types of members described above.  
By Lemma~\ref{lem:boundB},  
the set \(B\) cannot have cardinality \(1\), \(2\), or \(3\).
In contrast with the situation for \(r=3\), we do not obtain an analogue of  
Theorem~\ref{thm:n=3} when \(r=4\). We provide two examples with \( |B| = 4 \): in one case the augmentation submodule is simple, and in the other it is not.

For the choice 
\begin{align*}
B =\{
&\{\{\omega_1, \omega_2, \omega_3, \omega_4\},\{\omega_7, \omega_8, \omega_9, \omega_{10}\},\{\omega_5, \omega_{12}\},\{\omega_{11}, \omega_6\}\},\\
&\{\{\omega_4, \omega_5, \omega_6, \omega_7\},\{\omega_{10}, \omega_{11}, \omega_{12}, \omega_1\},\{\omega_2, \omega_9\},\{\omega_3, \omega_8\}\},\\
&\{\{\omega_1, \omega_2, \omega_3, \omega_4, \omega_9\},\{\omega_5, \omega_{10}\},\{\omega_7, \omega_{11}\},\{\omega_6, \omega_8, \omega_{12}\}\},\\
&\{\{\omega_2, \omega_5, \omega_8, \omega_{11}\},\{\omega_1, \omega_6, \omega_9\},\{\omega_3, \omega_7, \omega_{12}\},\{\omega_4, \omega_{10}\}\}\}
\end{align*}
the incidence matrix of the associated set system over \(\mathbb{C}\) has rank \(12\), whereas for the choice 
\begin{align*}
B =\{
&\{\{\omega_1, \omega_2, \omega_3, \omega_4, \omega_9\},\{\omega_5, \omega_{10}\},\{\omega_7, \omega_{11}\},\{\omega_6, \omega_8, \omega_{12}\}\},\\
&\{\{\omega_4, \omega_5, \omega_6, \omega_7, \omega_{12}\},\{\omega_1, \omega_8\},\{\omega_9, \omega_{11}, \omega_3\},\{\omega_2, \omega_{10}\}\},\\
&\{\{\omega_7, \omega_8, \omega_9, \omega_{10}, \omega_3\},\{\omega_4, \omega_{11}\},\{\omega_2, \omega_6, \omega_{12}\},\{\omega_1, \omega_5\}\},\\
&\{\{\omega_{10}, \omega_{11}, \omega_{12}, \omega_1, \omega_6\},\{\omega_2, \omega_7\},\{\omega_3, \omega_5, \omega_9\},\{\omega_4, \omega_8\}\}\}
\end{align*}
its rank is strictly smaller than \(12\). Hence these two subsets furnish the required examples.

%%%%%%%%%%%%%%%%%%%%%%%%%%%%%%%%%%%%%%%%%%%%%%%%%%%%%%%%%%%%%%%%%%%%%%%%%%%%%%%%%%%%%%%%%%%%%%%%%%%%%%%%%%%%%%%%
%%%%%%%%%%%%%%%%%%%%%%%%%%%%%%%%%%%%%%%%%%%%%%%%%%%%%%%%%%%%%%%%%%%%%%%%%%%%%%%%%%%%%%%%%%%%%%%%%%%%%%%%%%%%%%%%

\subsection{The case $r>4$}\label{r>4}
We extend the result to all integers \(r>4\) by a natural generalization of the case \(r=3\).

%We present the construction in two ways.
%The easier approach is based on viewing the diagram for $\mathcal{I}_r$ as a cycle consisting of $r$ edges, 
%where each edge contains $r$ vertices. 
%Each edge represents an element of $\mathcal{I}_r$, and every pair of consecutive edges intersects in exactly one vertex. We describe this approach in the following subsection. In this method, the set $\mathcal{I}_r$ consists only of these $r$ edges. Afterwards, we present the more difficult method that achieves the same goal. Both methods can be viewed as natural extensions of the case $r=3$.

%\subsubsection{First method}
%
%\subsubsection{Second method}
In this method, we again represent the set \(\mathcal{I}_r\) as a cycle whose edges
correspond to the elements of \(\mathcal{I}_r\), for every integer \(r > 4\).
In addition to these $r$ cyclic edges, the set \(\mathcal{I}_r\) also contains $r-3$ additional members, each of which intersects every edge of the cycle in exactly one vertex.

We choose the set 
%\begin{align*}
%\mathcal{I}_r = \big\{ 
%&\{\omega_1, \omega_2,\ldots, \omega_r\}, \{\omega_{r},\omega_{r+1},\ldots, \omega_{2r-1}\}, \{\omega_{2r-1}, \ldots, \omega_{3r-2}\},\ldots,\\ 
%&\{\omega_{(r-1)r-(r-2)},\omega_{(r-1)r-(r-2)+1},\ldots,\omega_{r(r-1)},\omega_1\},\\
%&\{\omega_{2},\omega_{r+1},\omega_{2r}\ldots,\omega_{(r-1)r-(r-2)+1}\},\\
%&\{\omega_{3},\omega_{r+2},\omega_{2r+1}\ldots,\omega_{(r-1)r-(r-2)+2}\},\\
%&\ldots\\
%&\{\omega_{r-2},\omega_{2r-3},\omega_{3r-4}\ldots,\omega_{r(r-1)-1}\} \big\}
%\end{align*}
%where \(\Omega_r = \{\omega_1, \ldots, \omega_{r(r-1)}\}\). 
\begin{align*}
\mathcal{I}_r = \big\{ 
&\{\omega_{J_r}, \omega_{c_{1,1}},\omega_{c_{1,2}},\ldots, \omega_{c_{1,r-3}},\omega_{P_1},\omega_{J_1}\}, \\
&\{\omega_{J_1}, \omega_{c_{2,1}},\ldots, \omega_{c_{2,r-3}},\omega_{P_2},\omega_{J_2}\},\\
&\{\omega_{J_2}, \omega_{c_{3,1}},\ldots, \omega_{c_{3,r-3}},\omega_{P_3}, \omega_{J_3}\},\\
&\vdots\\ 
&\{\omega_{J_{r-1}},\omega_{c_{r,1}},\ldots, \omega_{c_{r,r-3}},\omega_{P_r},\omega_{J_r}\},\\
&\{\omega_{c_{1,1}},\omega_{c_{2,1}},\ldots,\omega_{c_{r,1}}\},
%&\{\omega_{c_{1,2}},\omega_{c_{2,2}},\ldots,\omega_{c_{r,2}}\},\\
\ldots,
\{\omega_{c_{1,r-3}},\omega_{c_{2,r-3}},\ldots,\omega_{c_{r,r-3}}\} \big\}
\end{align*}
where \(\Omega_r = \{\omega_{J_1}, \ldots, \omega_{J_r},\omega_{P_1},\ldots,\omega_{P_r},\omega_{c_{1,1}},\ldots,\omega_{c_{r,r-3}}\}\). 

\begin{figure}[H]
\centering
\begin{tikzpicture}[scale=1.0, every node/.style={font=\small, anchor=south}]

%----- Main baseline (uniform y for top row of dots) -----
\def\topy{2.8}

%----- Top row dots -----
\foreach \x in {2.9,3.8,4.7,6.9,7.8,8.7} {
    \fill (\x,\topy) circle (2pt);
}

%----- Labels above top row dots -----
\node at (2.9,\topy+0.2) {$\omega_{J_r}$};
\node at (3.8,\topy+0.2) {$\omega_{c_{1,1}}$};
\node at (4.7,\topy+0.2) {$\omega_{c_{1,2}}$};
\node at (5.7,\topy+0.2) {$\ldots$};
\node at (6.9,\topy+0.2) {$\omega_{c_{1,r-3}}$};
\node at (7.8,\topy+0.2) {$\omega_{P_1}$};
\node at (8.7,\topy+0.2) {$\omega_{J_1}$};

%----- Middle row dots -----
\def\middley{0.5}
\foreach \x in {2.55,9} {
    \fill (\x,2.3) circle (2pt);
}
\node at (9.4,2.3) {$\omega_{c_{2,1}}$};
\node at (2.1,2.3) {$\omega_{P_r}$};

\foreach \x in {2.2,9.4} {
    \fill (\x,1.7) circle (2pt);
}
\node at (9.8,1.7) {$\omega_{c_{2,2}}$};
\node at (1.6,1.7) {$\omega_{c_{r,r-3}}$};
\node at (10.1,1.1) {$\ddots$};
\node at (1.4,1.15) {$\ddots$};

\def\bottomy{-2.5}
\foreach \x in {1.55,10} {
    \fill (\x,0.7) circle (2pt);
}
\node at (10.4,0.7) {$\omega_{c_{2,r-3}}$};
\node at (1.1,0.7) {$\omega_{c_{r,2}}$};

\foreach \x in {1.2,10.4} {
    \fill (\x,0.1) circle (2pt);
}
\node at (10.8,0.1) {$\omega_{P_2}$};
\node at (.7,0.1) {$\omega_{c_{r,1}}$};

\foreach \x in {.8,10.8} {
    \fill (\x,-0.5) circle (2pt);
}
\node at (11.2,-0.5) {$\omega_{J_2}$};
\node at (.3,-0.5) {$\omega_{J_{r-1}}$};
\node at (.8,-1.5) {$\vdots$};

\fill (10.93,-1.1) circle (2pt);\node at (11.33,-1.1) {$\omega_{c_{3,1}}$};
\fill (11.07,-1.7) circle (2pt);\node at (11.67,-1.7) {$\omega_{c_{3,2}}$};\node at (11.7,-2.2) {$\ddots$};
\fill (11.25,-2.6) circle (2pt);\node at (11.85,-2.6) {$\omega_{c_{3,r-3}}$};
\fill (11.36,-3.2) circle (2pt);\node at (11.96,-3.2) {$\omega_{P_3}$};
\fill (11.5,-3.8) circle (2pt);\node at (12.1,-3.8) {$\omega_{J_3}$};
\node at (11.5,-4.5) {$\vdots$};
%----- Main lines (reorganized for clean appearance) -----
\draw[black, line width=0.4pt] (.8,-0.5) -- (2.9,\topy) -- (8.7,\topy) --  (10.8,-0.5) -- (11.5,-3.8);

\draw[blue, line width=0.4pt] (1.1,-1.5) --(1.2,0.1) -- (3.8,\topy) -- (9,2.3) -- (10.93,-1.1) -- (10.7,-2.94);
\draw[red, line width=0.4pt] (1.3,-1) --(1.55,0.7) -- (4.7,\topy) -- (9.4,1.7) -- (11.07,-1.7) -- (11,-3.5);

\draw[magenta, line width=0.4pt] (1.5,-.7) --(2.2,1.7) -- (6.9,\topy) -- (10,0.7) -- (11.25,-2.6) -- (11.2,-4);
\end{tikzpicture}\caption{}
\label{fig:nGen}
\end{figure}

As in Figure~\ref{fig:n=3-1}, the set \(\mathcal{I}_r\) is illustrated again in
Figure~\ref{fig:nGen}.  
Since the subsets \(\{\omega_{c_{1,i}}, \omega_{c_{2,i}}, \ldots,
\omega_{c_{r,i}}\}\), for each \(1 \le i \le r-3\), cannot be represented by a
single line in the figure, we distinguish them using different colors.
The initial, terminal, and penultimate vertices of each edge in
Figure~\ref{fig:nGen} do not belong to any of these subsets.  
We refer to the initial and terminal vertices of each edge as the
\emph{joint vertices}, and to the penultimate vertex as the
\emph{penultimate vertex}.

To simplify the proof, we label the edges
\begin{align*}
&(\omega_{J_r}, \omega_{c_{1,1}},\ldots, \omega_{c_{1,r-3}},\omega_{P_1},\omega_{J_1}),\\
%(\omega_{J_1}, \omega_{c_{2,1}},\ldots, \omega_{c_{2,r-3}},\omega_{P_2},\omega_{J_2}),\\ 
&\vdots\\ 
&(\omega_{J_{r-1}},\omega_{c_{r,1}},\ldots, \omega_{c_{r,r-3}},\omega_{P_r},\omega_{J_r})
\end{align*}
by \(E_i\), according to their cyclic order.
For each edge \(E_i\), we denote by \(J_i\) the joint vertex at the end of the edge, and by \(P_i\) its penultimate vertex.

We discuss the two cases-\(r\) odd and \(r\) even-in separate subsections in
order to provide a subset \(\mathcal{B}_r\) of partitions of \(\Omega_r\) to construct the incidence matrix of a set system over
\(\mathbb{C}\) and show that it has rank \(r\).

%%%%%%%%%%%%%%%%%%%%%%%%%%%%%%%%%%%%%%%%%%%%%%%%%%%%%%%%%%%%%%%%%%%%%%%%%%
%%%%%%%%%%%%%%%%%%%%%%%%%%%%%%%%%%%%%%%%%%%%%%%%%%%%%%%%%%%%%%%%%%%%%%%%%%

\subsubsection{(\textbf{\(r \) Odd})}\label{rOdd}
The elements of \(\mathcal{B}_r\) are illustrated in Table~\ref{Table1}.
The first row consists of vertices from Figure~\ref{fig:nGen} corresponding to the vertices \(\omega_{c_{i,1}}\).
The subsequent rows follow the same pattern, for the penultimate row, which corresponds to the penultimate vertices, and the last row, which corresponds to the joint vertices.
This table summarizes our partitions.

We define $(r-1)$ partitions corresponding to the first edge 
\[(\omega_{J_r}, \omega_{c_{1,1}},\omega_{c_{1,2}},\ldots, \omega_{c_{1,r-3}},\omega_{P_1},\omega_{J_1})\]
up to 
\[(\omega_{J_{r-2}},\omega_{c_{r-1,1}},\ldots, \omega_{c_{r-1,r-3}},\omega_{P_{r-1}},\omega_{J_{r-1}}),\] 
and an additional partition. We denote these partitions by $B_1$ through $B_{r-1}$, and the last one by $B_r^{\ast}$.

\begin{figure}[H]
\centering
\begin{tikzpicture}[x=1cm,y=1cm]

% -------------------------- TABLE --------------------------
\node at (6, -2) {
%\begin{tabular}{lllllllll}
%$\omega_2$     & $\omega_{r+1}$  & $\omega_{2r}$   & $\cdots$ & $\cdots$ & $\cdots$ & $\cdots$ & $\cdots$ &$\omega_{\substack{(r-1)r\\-(r-2)+1}}$\\
%$\omega_3$     & $\omega_{r+2}$  & $\omega_{2r+1}$ & $\cdots$ & $\cdots$ & $\cdots$ & $\cdots$ & $\cdots$ &$\omega_{\substack{(r-1)r\\-(r-2)+2}}$\\
%$\vdots$       & $\vdots$        & $\vdots$        & $\cdots$ & $\cdots$ & $\cdots$ & $\cdots$ & $\cdots$ & $\vdots$\\
%$\omega_{r-2}$ & $\omega_{2r-3}$ & $\omega_{3n-4}$ & $\cdots$ & $\cdots$ & $\cdots$ & $\cdots$ & $\cdots$ & $\omega_{(r-1)r-1}$\\
%$\omega_{P_1}$ & $\omega_{P_2}$   & $\omega_{P_3}$  & $\cdots$ & $\cdots$ & $\cdots$ & $\cdots$ & $\cdots$ & $\omega_{P_r}$\\
%$\omega_{J_1}$ & $\omega_{J_2}$  & $\omega_{J_3}$  & $\omega_{J_4}$ & $\omega_{J_5}$& $\cdots$ & $\omega_{J_{r-2}}$
%               & $\omega_{J_{r-1}}$
%               & $\omega_{J_r}(=\omega_1)$
%\end{tabular}
%};
\begin{tabular}{lllllllll}
$\omega_{c_{1,1}}$   & $\omega_{c_{2,1}}$   & $\omega_{c_{3,1}}$   & $\cdots$ & $\cdots$ & $\cdots$ & $\cdots$ & $\cdots$ & $\omega_{c_{r,1}}$\\
$\omega_{c_{1,2}}$   & $\omega_{c_{2,2}}$   & $\omega_{c_{3,2}}$   & $\cdots$ & $\cdots$ & $\cdots$ & $\cdots$ & $\cdots$ & $\omega_{c_{r,2}}$\\
$\vdots$             & $\vdots$             & $\vdots$             & $\vdots$ & $\vdots$ & $\vdots$ & $\vdots$ & $\vdots$ & $\vdots$\\
$\omega_{c_{1,r-3}}$ & $\omega_{c_{2,r-3}}$ & $\omega_{c_{3,r-3}}$ & $\cdots$ & $\cdots$ & $\cdots$ & $\cdots$ & $\cdots$ & $\omega_{c_{r,r-3}}$\\
$\omega_{P_1}$       & $\omega_{P_2}$       & $\omega_{P_3}$       & $\cdots$ & $\cdots$ & $\cdots$ & $\cdots$ & $\cdots$ & $\omega_{P_r}$\\
$\omega_{J_1}$       & $\omega_{J_2}$       & $\omega_{J_3}$       & $\omega_{J_4}$ & $\omega_{J_5}$& $\cdots$ & $\omega_{J_{r-2}}$
  & $\omega_{J_{r-1}}$  & $\omega_{J_r}$
\end{tabular}
};

% -------------------------- RECTANGLES --------------------------
% Same coordinates as PSTricks
\draw[cyan, thick] (1.8,-3.6) rectangle (.4,-.5);
\draw[cyan, thick] (3.4,-3.6) rectangle (4.2,-3.05);
\draw[cyan, thick] (5.8,-3.6) rectangle (6.5,-3.05);
\draw[cyan, thick] (7.6,-3.6) rectangle (8.7,-3.05);
\draw[cyan, thick] (10.15,-3.6) rectangle (11,-3.05);

\end{tikzpicture}\renewcommand{\figurename}{Table}\caption{}\label{Table1}
\end{figure}

For $B_1$, we place the elements of the corresponding edge 
\[(\omega_{J_r}, \omega_{c_{1,1}},\omega_{c_{1,2}},\ldots, \omega_{c_{1,r-3}},\omega_{P_1},\omega_{J_1})\]
in the first part of the partition, together with the joint vertices of the figure such that there is a two-edge distance between them containing $\omega_{J_r}$ and $\omega_{J_1}$.
The fact that $r$ is odd guarantees this ordering. In Table~\ref{Table1}, we highlight these elements with a blue rectangle:
\[
{B_1}_1 = \{\omega_{J_r}, \omega_{c_{1,1}},\omega_{c_{1,2}},\ldots, \omega_{c_{1,r-3}},\omega_{P_1},\omega_{J_1}, \omega_{J_3},\ldots,\omega_{J_{r-2}}\}.
\]
Now, we omit the elements of the first part of the partition in our table, removing the first column and some joint vertices. For each joint vertex that has been removed, we place it together with the penultimate vertex above it, leaving the corresponding penultimate vertex cell blank. This process is illustrated in Table~\ref{Table2}.

\begin{figure}[H]
\centering
\begin{tikzpicture}[x=1cm,y=1cm]

% -------------------------- TABLE --------------------------
\node at (6, -2) {
%\begin{tabular}{llllllll}
%$\cc{\omega_{r+1}}$  & $\yy{\omega_{2r}}$   & $\cdots$ & $\cdots$ & $\cdots$ & $\cdots$ & $\cdots$ &$\bb{\omega_{\substack{(r-1)r\\-(r-2)+1}}}$\\
%$\yy{\omega_{r+2}}$  & $\omega_{2r+1}$ & $\cdots$ & $\cdots$ & $\cdots$ & $\cdots$ & $\cdots$ &$\cc{\omega_{\substack{(r-1)r\\-(r-2)+2}}}$\\
%$\vdots$        & $\vdots$        & $\cdots$ & $\bb{\omega_{5r-7}}$ & $\cdots$ & $\cdots$ & $\cdots$ & $\vdots$\\
%$\omega_{2r-3}$ & $\omega_{3r-4}$ & $\bb{\omega_{4r-5}}$ &  $\cc{\omega_{5r-6}}$ & $\cdots$ & $\omega_{\substack{(r-2)r\\ -(r-1)}}$ & $\cdots$ & $\omega_{(r-1)r-1}$\\
%$\omega_{P_2}$ & blk      & $\cc{\omega_{P_4}}$ &blk& $\cdots$ &blk& $\cdots$ &blk \\
%$\bb{\omega_{J_2}}$ & $\cc{\omega_{P_3}}$ & $\yy{\omega_{J_4}}$ 
%                & $\omega_{P_5}$ 
%                & $\cdots$
%                & $\omega_{P_{r-2}}$
%                & $\omega_{J_{r-1}}$
%                & $\omega_{P_{r}}$\\
%\end{tabular}
\begin{tabular}{llllllll}
$\cc{\omega_{c_{2,1}}}$  & $\yy{\omega_{c_{3,1}}}$   & $\cdots$ & $\cdots$ & $\cdots$ & $\cdots$ & $\cdots$ &$\bb{\omega_{c_{r,1}}}$\\
$\yy{\omega_{c_{2,2}}}$  & $\omega_{c_{3,2}}$ & $\cdots$ & $\cdots$ & $\cdots$ & $\cdots$ & $\cdots$ &$\cc{\omega_{c_{r,2}}}$\\
$\vdots$        & $\vdots$        & $\vdots$ & $\bb{\omega_{c_{5,r-4}}}$ & $\vdots$ & $\vdots$ & $\vdots$ & $\vdots$\\
$\omega_{c_{2,r-3}}$ & $\omega_{c_{3,r-3}}$ & $\bb{\omega_{c_{4,r-3}}}$ &  $\cc{\omega_{c_{5,r-3}}}$ & $\cdots$ & $\omega_{c_{r-2,r-3}}$ & $\omega_{c_{r-1,r-3}}$ & $\omega_{c_{r,r-3}}$\\
$\omega_{P_2}$ & blk      & $\cc{\omega_{P_4}}$ &blk& $\cdots$ &blk& $\omega_{P_{r-1}}$ &blk \\
$\bb{\omega_{J_2}}$ & $\cc{\omega_{P_3}}$ & $\yy{\omega_{J_4}}$ 
                & $\omega_{P_5}$ 
                & $\cdots$
                & $\omega_{P_{r-2}}$
                & $\omega_{J_{r-1}}$
                & $\omega_{P_{r}}$\\
\end{tabular}
};
%\draw[blue] (.8,-3.75) -- (5,-1.4);
%\draw[blue] (5,-1.4) -- (10,0);
%\draw[red] (2.25,-3.75) -- (6.1,-1.6);
%\draw[red] (6.1,-1.6) -- (11.5,0);

\end{tikzpicture}\renewcommand{\figurename}{Table}\caption{}\label{Table2}
\end{figure}

The other parts of the partition are constructed from the cyclic diagonals of the new table, which has $(r-1)$ rows and columns. 
For any position $(i,j)$ in this table, we define the diagonal through $(i,j)$ by
\[
D(i,j)=\{\, (i+k,\; j-k) \pmod{(r-1)} : k \in \mathbb{Z} \,\},
\]
where the modulo operation is taken with respect to the number of rows and columns.
Equivalently, we may view the table as wrapped around a cylinder, so that the diagonals run cyclically from the lower left to the upper right.
Because we omit the entries in the odd positions of the last row of Table~\ref{Table1}, and because the vertices in the penultimate row shift upward to occupy their places, the diagonals that begin in even positions of the last row behave differently from those that begin in odd positions.
If we start from an even position in the last row, the first cell on that diagonal corresponds to a joint vertex; after a blank cell, each subsequent step upward moves one row and one column, which corresponds exactly to traversing the edges of Figure~\ref{fig:nGen} cyclically (ignoring the omitted column if it appears along the way).
If we start from an odd position in the last row, then the first and penultimate entries of that diagonal correspond to the penultimate vertices of two neighbouring edges. This is the only difference between diagonals starting in even and odd positions.

Since the first part of the partition contains all the elements of the first edge, its other elements are joint vertices at a two-edge distance, and all rows of the table except the last contain at most one element, this part satisfies the \(\J\)-minimal compatibility condition. 
Furthermore, the other parts of the partition, except for the penultimate-vertex row, contain at most one element per row, as do the columns representing the edges of Figure~\ref{fig:nGen}. Hence, $B_1$ satisfies the \(\J\)-minimal compatibility condition.

Moreover, in the same way, we may define a partition with respect to another edge of Figure~\ref{fig:nGen}, focusing on the corresponding edge in Table~\ref{Table1}.

Assume that we define a partition corresponding to omitting edge~$j$ in Table~\ref{Table1}, 
that is, the column $j$ together some joint vertices are omitted, and we construct a new table analogous to Table~\ref{Table2}. 
Consider this new table. 
Let \(\mathrm{Dign}_{i,j}\) denote the set of entries along the diagonal above the penultimate row (that is, excluding the joint and penultimate vertices), whose first entry occurs in column~\(i\), namely \(\omega_{c_{i,r-3}}\).
Note that column $j$ in Table~\ref{Table1} does not appear along the path of this diagonal.
If the diagonal begins in the last row at a position containing a joint vertex 
of the figure, then its total contribution is the joint vertex in column~$i$ 
together with $\mathrm{Dign}_{i+2,j}$.
In contrast, if the diagonal begins at a penultimate shifted vertex, then its 
contribution consists of the penultimate vertex of edge~$i$, the penultimate 
vertex of edge~$i+1$, and again $\mathrm{Dign}_{i+2,j}$.

We now define the final partition $B_r^{\ast}$ as follows. The first part, and each part up to the one preceding the penultimate, consists respectively of the first row up to the row just before the penultimate row of Table~\ref{Table1}. Equivalently, each of these parts corresponds to one of the colored edges of $\mathcal{I}_r$ in Figure~\ref{fig:nGen}. The remaining three parts form a partition of the penultimate and joint vertices.
We place the elements 
$\omega_{J_r}, \omega_{J_2}, \omega_{J_4}, \ldots, \omega_{J_{r-3}}$ 
in the first part, 
the elements 
$\omega_{J_1}, \omega_{J_3}, \ldots, \omega_{J_{r-2}}$ 
in the second part, 
and the remaining joint vertex 
$\omega_{J_{r-1}}$ 
in the last part. We now divide the penultimate vertices among these parts. 
We place the penultimate vertex $\omega_{P_{r-1}}$ in the first part, 
the vertex $\omega_{P_{r}}$ in the second part, 
and all remaining penultimate vertices in the last part. 
In this way, the partition $B_r^{\ast}$ satisfies the \(\J\)-minimal compatibility condition.

Therefore, we obtain the desired set  
\[
\mathcal{B}_r=\{\, B_1,\ \ldots,\ B_{r-1},\ B_r^{\ast} \,\}.
\]

Now, we verify that $\mathcal{W}^{\perp} = 0$; hence the augmentation submodule $\Aug(\mathbb{C}\Omega_r)$ is simple.

Suppose, to the contrary, that \( \mathcal{W}^{\perp} \neq 0 \).  
Then there exist integers \( \alpha_{J_1}, \ldots, \alpha_{J_r}, \alpha_{P_1}, \ldots, \alpha_{P_r},\alpha_{c_{1,1}},\ldots,\alpha_{c_{r,r-3}} \) such that  
\begin{align*}
A = &\alpha_{J_1}\omega_{J_1}+ \cdots+ \alpha_{J_r}\omega_{J_r}+\\
&\alpha_{P_1}\omega_{P_1}+ \cdots+ \alpha_{P_r}\omega_{P_r}+
 \alpha_{c_{1,1}}\omega_{c_{1,1}}+\cdots+\alpha_{c_{r,r-3}}\omega_{c_{r,r-3}}\in \mathcal{W}^{\perp}.
\end{align*}

%The joint vertex corresponding to column~$i$ in Table~\ref{Table1} is 
%\(\omega_{J_i}\).  
Let \(1 \le j \le r-3\).
In \(B_j\), we obtain the relation  
\[
\alpha_{J_{j+1}} + 1_{\mathrm{Dign}_{j+3,j}}(A) = 0,
\]
and in \(B_{j+2}\), we have
\[
\alpha_{J_j} + 1_{\mathrm{Dign}_{j+3,j+2}}(A) = 0.
\]
Since \(\mathrm{Dign}_{j+3,j} = \mathrm{Dign}_{j+3,j+2}\), it follows that
\[
\alpha_{J_{j+1}} = \alpha_{J_j}.
\]
Hence, we have \begin{equation}\label{eqeq01}\alpha_{J_1}=\alpha_{J_2}=\cdots=\alpha_{J_{r-2}}.\end{equation}

Also, in \(B_1\), we obtain the relation  
\[
\alpha_{J_{r-1}} + 1_{\mathrm{Dign}_{2,1}}(A) = 0,
\]
and in \(B_{r-1}\), we have
\[
\alpha_{J_r} + 1_{\mathrm{Dign}_{2,r-1}}(A) = 0.
\]
Again as \(\mathrm{Dign}_{2,1} = \mathrm{Dign}_{2,r-1}\), it follows that
\begin{equation}\label{eqeq02}
\alpha_{J_{r-1}} = \alpha_{J_r}.
\end{equation}

Now, form the partition $B_r^{\ast}$, we obtain the relations
\begin{equation}\label{eqeq}
\begin{aligned}
\alpha_{c_{1,1}}+\alpha_{c_{2,1}}+\cdots+\alpha_{c_{r,1}}&=0,\\
\alpha_{c_{1,2}}+\alpha_{c_{2,2}}+\cdots+\alpha_{c_{r,2}}&=0,\\
&\vdots\\
\alpha_{c_{1,r-3}}+\alpha_{c_{2,r-3}}+\cdots+\alpha_{c_{r,r-3}}&=0.
\end{aligned}
\end{equation}

As before, we may view these equations as arising from a table wrapped around a cylinder, so that the diagonals run cyclically from the lower left to the upper right.
There is a natural correspondence between the diagonals appearing in these equations and the diagonals of Table~\ref{Table1} that lie above the penultimate row, once a column is omitted.
In particular,
each diagonal of Table~\ref{Table1} (with one column removed) corresponds to a joint vertex of the table, provided that the omitted column occurs immediately before the starting position of the diagonal or lies before the joint vertex.
More precisely, in the first case this happens when two columns precede the starting position of the diagonal and the joint vertex follows (see the left-hand side of Table~\ref{Table3}), whereas in the second case exactly one column separates the joint vertex from the starting position of the diagonal (see the right-hand side of Table~\ref{Table3}).
\begin{figure}[H]
\centering
\begin{tikzpicture}[x=1cm,y=1cm]
\node at (6, -2) {
\begin{tabular}{lllll@{\hspace{1em}\vrule\hspace{1em}}llll}
$\vdots$            & $\vdots$           & $\vdots$ & $\vdots$                    & $\vdots$                    &
$\vdots$ & $\vdots$                & $\vdots$           & $\vdots$                                              \\
$\vdots$            & $\vdots$           & $\vdots$ & $\cdots$                    & $\bb{\omega_{c_{j+4,r-2}}}$ &
$\vdots$ & $\vdots$                & $\vdots$           & $\vdots$                                              \\
$\cdots$            & $\cdots$           & blk      & $\bb{\omega_{c_{j+3,r-3}}}$ & $\cdots$                    &                  
blk      & $\cdots$                & $\cdots$           & $\bb{\omega_{c_{j+3,r-3}}}$                           \\
$\omega_{P_j}$      & blk                & blk      & $\omega_{P_{j+3}}$          & $\cdots$                    &                  
blk      & $\omega_{P_{j+1}}$      & blk                & $\omega_{P_{j+3}}$                                    \\
$\bb{\omega_{J_j}}$ & $\omega_{P_{j+1}}$ & blk      & $\omega_{J_{j+3}}$          & $\cdots$                    &
blk      & $\bb{\omega_{J_{j+1}}}$ & $\omega_{P_{j+2}}$ & $\omega_{J_{j+3}}$                          
\end{tabular}};
\end{tikzpicture}\renewcommand{\figurename}{Table}\caption{}\label{Table3}
\end{figure}

These diagonals split into two types. 
The first type consists of three diagonals which, when combined with 
\(\omega_{J_r}\) and \(\omega_{J_{r-1}}\), form exactly the following three 
pairwise distinct subsets:
\[
\mathrm{Dign}_{2,1},\mathrm{Dign}_{3,2},\mathrm{Dign}_{1,r-2},
\]
(\(\mathrm{Dign}_{2,1} \)and \( \mathrm{Dign}_{2,r-1}\) are identical) for which with $\omega_{J_r}$ or $\omega_{J_{r-1}}$ included in some part of a partition $B_1$, $B_2$, $B_{r-2}$, or $B_{r-1}$,
while the remaining ones correspond to one of the sets
\[
\mathrm{Dign}_{j+3,j}=\mathrm{Dign}_{j+3,j+2}, \qquad 1 \leq j \leq r-3
\]
for which with $\omega\in\Omega\setminus\{\omega_{J_r},\omega_{J_{r-1}}\}$ included in some part of a partition $B_i$.
Since each diagonal admits two possible choices, we can always avoid omitting the $r$-th column.
Consequently, using (\ref{eqeq01}) and (\ref{eqeq02}), we obtain the relation
\begin{equation}\label{Eq1}
3\alpha_{J_r} + (r-3)\alpha_{J_1} = 0.
\end{equation}

Again, we interpret equations~\eqref{eqeq} as arising from a table wrapped around a cylinder, 
so that its diagonals run cyclically.
Each diagonal of Table~\ref{Table1}, after omitting one column, corresponds either to a joint vertex 
or to two penultimate vertices of the table.
More precisely, a diagonal corresponds to a joint vertex if the omitted column occurs immediately 
before the starting position of the diagonal, with exactly two columns preceding the start and the 
joint vertex appearing next (see the left-hand side of Table~\ref{Table4}).
Also, the diagonal corresponds to two penultimate vertices: one located immediately before the 
starting position and the other two positions before it, with the omitted column lying between these 
two penultimate vertices (see the right-hand side of Table~\ref{Table4}).
\begin{figure}[H]
\centering
\begin{tikzpicture}[x=1cm,y=1cm]
\node at (6, -2) {
\begin{tabular}{lllll@{\hspace{1em}\vrule\hspace{1em}}llll}
$\vdots$            & $\vdots$           & $\vdots$ & $\vdots$                    & $\vdots$                    &
$\vdots$ & $\vdots$                & $\vdots$           & $\vdots$                                              \\
$\vdots$            & $\vdots$           & $\vdots$ & $\cdots$                    & $\bb{\omega_{c_{j+1,r-2}}}$ &
$\vdots$ & $\vdots$                & $\vdots$           & $\vdots$                                              \\
$\cdots$            & $\cdots$           & blk      & $\bb{\omega_{c_{j,r-3}}}$ & $\cdots$                    &                  
$\cdots$      & blk                & $\cdots$           & $\bb{\omega_{c_{j,r-3}}}$                           \\
$\omega_{P_{j-3}}$      & blk                & blk      & $\omega_{P_{j}}$          & $\cdots$                    &                  
blk    & blk      & $\bb{\omega_{P_{j-1}}}$                &blk                                    \\
$\bb{\omega_{J_{j-3}}}$ & $\omega_{P_{j-2}}$ & blk      & $\omega_{J_{j}}$          & $\cdots$                    &
$\bb{\omega_{P_{j-3}}}$     & blk & $\omega_{J_{j-1}}$ & $\omega_{P_{j}}$                          
\end{tabular}};
\end{tikzpicture}\renewcommand{\figurename}{Table}\caption{}\label{Table4}
\end{figure}

Then we have
\[
\alpha_{P_{j-3}} + \alpha_{P_{j-1}} = \alpha_{J_{j-3}}
\quad \text{with indices taken in } \mathbb{Z}_{r}.
\]
Indeed,
\[
\{\omega_{J_{j-3}}\} \cup \mathrm{Dign}_{j,j-1} \in B_{j-1}
\quad \text{and} \quad
\{\omega_{P_{j-3}}, \omega_{P_{j-1}}\} \cup \mathrm{Dign}_{j,j-2} \in B_{j-2},
\]
and the diagonals 
\[
\mathrm{Dign}_{j,j-2} \quad \text{and} \quad \mathrm{Dign}_{j,j-1}
\]
are identical. Hence, from the partitions $B_1$ to $B_{r-1}$, we obtain the relations
\begin{equation}\label{eqeq2}
\begin{aligned}
\alpha_{P_{r}} + \alpha_{P_{2}} &= \alpha_{J_{r}},\\
\alpha_{P_{1}} + \alpha_{P_{3}} &= \alpha_{J_{1}},\\
&\vdots\\
\alpha_{P_{r-3}} + \alpha_{P_{r-1}} &= \alpha_{J_{r-3}}.
\end{aligned}
\end{equation}

Moreover, we may repeat this argument at the stage before omitting one column. 
\begin{figure}[H]
\centering
\begin{tikzpicture}[x=1cm,y=1cm]
\node at (6, -2) {
\begin{tabular}{lllll@{\hspace{1em}\vrule\hspace{1em}}lllll}
$\cdots$                   & $\bb{\omega_{c_{j+1,1}}}$  & blk      & $\cdots$                & $\cdots$           &  
$\cdots$                   & $\bb{\omega_{c_{j+1,1}}}$  & $\cdots$      & blk                & $\cdots$ 
\\
$\bb{\omega_{c_{j,2}}}$    & $\cdots$                   & blk      & $\cdots$                & $\cdots$           &  
$\bb{\omega_{c_{j,2}}}$    & $\cdots$                   & $\cdots$      & blk                & $\cdots$
\\
$\vdots$                   & $\vdots$                   & $\vdots$ & $\vdots$                & $\vdots$           &  
$\vdots$                   & $\vdots$                   & $\vdots$ & $\vdots$                & $\vdots$ 
\\
$\cdots$                   & $\cdots$                   & blk      & $\cdots$                & $\cdots$           &  
$\cdots$                   & $\cdots$                   & $\cdots$      & blk                & $\cdots$
\\
$\cdots$                   & blk                        & blk      & $\omega_{P_{j+3}}$      & blk                &  
$\cdots$                   & $\cdots$                        & blk      & blk      & $\bb{\omega_{P_{j+4}}}$  
\\
$\cdots$                   & $\omega_{P_{j+1}}$         & blk      & $\bb{\omega_{J_{j+3}}}$ & $\omega_{P_{j+4}}$ &      
$\cdots$                   & $\cdots$         & $\bb{\omega_{P_{j+2}}}$      & blk & $\omega_{J_{j+4}}$             
\end{tabular}};
\end{tikzpicture}\renewcommand{\figurename}{Table}\caption{}\label{Table5}
\end{figure}
In this case, we obtain
\[
\{\omega_{J_{j+3}}\} \cup \mathrm{Dign}_{j+5,j+2} \in B_{j+2}
\quad \text{and} \quad
\{\omega_{P_{j+2}}, \omega_{P_{j+4}}\} \cup \mathrm{Dign}_{j+5,j+3} \in B_{j+3}.
\]
The diagonals
\[
\mathrm{Dign}_{j+5,j+3} \quad \text{and} \quad \mathrm{Dign}_{j+5,j+2}
\]
are identical. Consequently, we obtain
\[\alpha_{P_{j+2}}+ \alpha_{P_{j+4}}=\alpha_{J_{j+3}}.\]
Hence, from the partitions $B_1$ to $B_{r-1}$, we obtain the relations
\begin{equation}\label{eqeq3}
\begin{aligned}
\alpha_{P_{1}} + \alpha_{P_{3}} &= \alpha_{J_{2}},\\
\alpha_{P_{2}} + \alpha_{P_{4}} &= \alpha_{J_{3}},\\
\alpha_{P_{3}} + \alpha_{P_{5}} &= \alpha_{J_{4}},\\
&\vdots\\
\alpha_{P_{r-2}} + \alpha_{P_{r}} &= \alpha_{J_{r-1}}.
\end{aligned}
\end{equation}
From relations~\eqref{eqeq2} and~\eqref{eqeq3}, we obtain
\[
\alpha_{P_{r}} + \alpha_{P_{2}} = \alpha_{J_{r}}
\quad \text{and} \quad
\alpha_{P_{r-2}} + \alpha_{P_{r}} = \alpha_{J_{r-1}}.
\]
Now, it follows that from relation~\eqref{eqeq02} that \(\alpha_{P_{2}} = \alpha_{P_{r-2}}\).
Hence, since
\[
\alpha_{P_{2}} + \alpha_{P_{4}} = \alpha_{J_{3}}
\quad \text{and} \quad
\alpha_{P_{r-4}} + \alpha_{P_{r-2}} = \alpha_{J_{r-4}},
\]
we obtain \(\alpha_{P_{4}} = \alpha_{P_{r-4}}\).
Similarly, we obtain
\[
\alpha_{P_{k}} = \alpha_{P_{k+4}}, \quad \text{for } 1 \leq k \text{ and } k+4 \leq r-1.
\]
We conclude that, if \(r = 4r' + 1\), then
\begin{equation}\label{eqeq4}
\begin{aligned}
&\alpha_{P_{2}} = \alpha_{P_{6}} = \alpha_{P_{10}} = \cdots = \alpha_{P_{r-3}}
= \alpha_{P_{r-2}} = \alpha_{P_{r-6}} = \cdots = \alpha_{P_{7}} = \alpha_{P_{3}},\\
\text{and}\\
&\alpha_{P_{4}} = \alpha_{P_{8}} = \alpha_{P_{12}} = \cdots = \alpha_{P_{r-1}}
= \alpha_{P_{r-4}} = \alpha_{P_{r-8}} = \cdots = \alpha_{P_{5}} = \alpha_{P_{1}}.
\end{aligned}
\end{equation}
Otherwise, if \(r = 4r' + 3\), then
\[
\alpha_{P_{2}} = \alpha_{P_{6}} = \alpha_{P_{10}} = \cdots = \alpha_{P_{r-1}}
= \alpha_{P_{r-2}} = \alpha_{P_{r-6}} = \cdots = \alpha_{P_{5}} = \alpha_{P_{1}},
\]
and
\[
\alpha_{P_{4}} = \alpha_{P_{8}} = \alpha_{P_{12}} = \cdots = \alpha_{P_{r-3}}
= \alpha_{P_{r-4}} = \alpha_{P_{r-8}} = \cdots = \alpha_{P_{7}} = \alpha_{P_{3}}.
\]

Moreover, from the partition $B_r^{\ast}$, we obtain 
\begin{equation}\label{eqeq5}
\begin{aligned}
&\alpha_{J_{r}}+ \alpha_{J_2}+ \alpha_{J_4}+ \cdots+ \alpha_{J_{r-5}}+\alpha_{J_{r-3}}+\alpha_{P_{r-1}}=0,\\ 
&\alpha_{J_{1}}+ \alpha_{J_{3}}+ \cdots+ \alpha_{J_{r-4}}+\alpha_{J_{r-2}}+\alpha_{P_{r}}=0,\\
&\alpha_{J_{r-1}}+\alpha_{P_{1}}+\alpha_{P_{2}}+\cdots+\alpha_{P_{r-2}}=0.
\end{aligned}
\end{equation}
Then, combining the second relation in~\eqref{eqeq5} with relation~\eqref{eqeq01}, we obtain
\[
\frac{r-1}{2}\,\alpha_{J_{1}} + \alpha_{P_{r}} = 0.
\]
And combining the third relation in~\eqref{eqeq5} with relations~\eqref{eqeq3} and~\eqref{eqeq4}, we obtain
\[
\alpha_{J_{r}}+\frac{r-3}{2}\,\alpha_{J_{1}} + \alpha_{P_{r-2}}=0.
\]
Taking the sum of the two relations yields
\[
2\alpha_{J_{r}} + (r-2)\alpha_{J_{1}} = 0.
\]
Therefore, it follows from relation~\eqref{Eq1} that \(\alpha_{J_{r}} = \alpha_{J_{1}} = 0\).
%\begin{equation}\label{Eq1}
%3\alpha_{J_r} + (r-3)\alpha_{J_1} = 0.
%\end{equation}
Now, from the first and second relations in~\eqref{eqeq5}, we obtain
\(\alpha_{P_{r-1}} = \alpha_{P_{r}} = 0\).
Moreover, using relation~\eqref{eqeq3}, it follows that
\[
\alpha_{P_{1}} = \alpha_{P_{2}} = \cdots = \alpha_{P_{r-2}} = 0.
\]

Now, since the coefficients of the joint and penultimate vertices are zero, we proceed to compute the coefficients of the remaining vertices.
Then, for \(2 \leq j \leq r-1\), we have
\[
1_{\mathrm{Dign}_{j-1,j}}(A) = 1_{\mathrm{Dign}_{j,j-1}}(A) = 0.
\]
Consequently, \(\alpha_{c_{j-1,r-3}} = \alpha_{c_{j,r-3}}\) for \(2 \leq j \leq r-1\).
It follows that
\[
\alpha_{c_{1,r-3}} = \alpha_{c_{2,r-3}} = \alpha_{c_{3,r-3}} = \cdots = \alpha_{c_{r-1,r-3}}.
\]
Moreover, since
\[
1_{\mathrm{Dign}_{r,1}}(A) = 1_{\mathrm{Dign}_{1,r-1}}(A) = 0,
\]
we also have
\[
\alpha_{c_{r,r-3}} = \alpha_{c_{1,r-3}}.
\]
Now, from the last relation in~\eqref{eqeq}, we deduce that \(\alpha_{c_{1,r-3}} = 0\). Proceeding in the same manner, we conclude that 
\[
\alpha_{c_{i,j}} = 0 \quad \text{for all } 1 \leq i \leq r \text{ and } 1 \leq j \leq r-3.
\] 
This implies that \(A = 0\), and hence \(\mathcal{W}^{\perp} = 0\).  The result follows.

Let us define the partition \(B_r\) in a manner similar to the other partitions \(B_1, \ldots, B_{r-1}\).
Now, we show that for
\[
\mathcal{B}'_r = \{ B_1, \ldots, B_{r-1}, B_r \},
\]
we have \(\mathcal{W}^{\perp} \neq 0\). Consequently, unlike the case \(r = 3\),
for the monoid \(M\) constructed from \(\mathcal{B}'_r\) and \(\mathcal{I}_r\) as in
\eqref{ConsM}, the augmentation submodule \(\Aug(\mathbb{C}\Omega_r)\) is not
simple.

Let
\[
A = \alpha_{c_{1,1}} \omega_{c_{1,1}} + \cdots + c_{i,j} \omega_{c_{i,j}}+ \cdots \alpha_{c_{r,r-3}} \omega_{c_{r,r-3}},
\]
where
\[
\alpha_{c_{i,j}} =
\begin{cases}
1, & \text{if \(j\) is odd},\\
-1, & \text{if \(j\) is even},
\end{cases}
\qquad \text{for all } 1 \le i \le r.
\]
For every \(1 \le k \le r\), each block of the partition \(B_k \in \mathcal{B}'_r\)
contains \(\tfrac{r-3}{2}\) vertices \(\omega_{c_{i,j}}\) with \(j\) odd and
\(\tfrac{r-3}{2}\) vertices \(\omega_{c_{i,j}}\) with \(j\) even.
It follows that \(1_{B_{kl}}(A)=0\) for all \(1 \le \ell \le r\).
Hence \(A \in \mathcal{W}^{\perp}\), and the result follows.

%%%%%%%%%%%%%%%%%%%%%%%%%%%%%%%%%%%%%%%%%%%%%%%%%%%%%%%%%%%%%%%%%%%%%%%%%%
%%%%%%%%%%%%%%%%%%%%%%%%%%%%%%%%%%%%%%%%%%%%%%%%%%%%%%%%%%%%%%%%%%%%%%%%%%

\subsubsection{(\textbf{\(r \) Even})}\label{rEven}
When \(n\) is even, there are some differences in the illustration of the elements of
\(\mathcal{B}_r\).
As in Subsection~\ref{rOdd}, we use a table analogous to Table~\ref{Table1}, namely
Table~\ref{Table1-e}.
Again, the first row consists of the vertices from Figure~\ref{fig:nGen}
corresponding to \(\omega_{c_{i,1}}\), followed successively by the vertices
\(\omega_{c_{i,j}}\) for \(2 \le j \le r-3\).
The penultimate row corresponds to the penultimate vertices, and the final row
corresponds to the joint vertices.

We describe the partition \(B_1\), corresponding to the first edge of the figure
\[
(\omega_{J_r}, \omega_{c_{1,1}}, \omega_{c_{1,2}}, \ldots,
\omega_{c_{1,r-3}}, \omega_{P_1}, \omega_{J_1}).
\]
The partitions \(B_2,\ldots,B_{r-1}\) are obtained in an analogous way,
corresponding to the successive edges of the cycle in
Figure~\ref{fig:nGen}.  
As in Subsection~\ref{rOdd}, we define the final partition and denote it by
\(B_r^{\ast}\).

\begin{figure}[H]
\centering
\begin{tikzpicture}[x=1cm,y=1cm]

% -------------------------- TABLE --------------------------
\node at (6, -2) {
\begin{tabular}{llllllllll}
$\omega_{c_{1,1}}$   & $\omega_{c_{2,1}}$   & $\omega_{c_{3,1}}$   & $\cdots$ & $\cdots$ & $\cdots$ & $\cdots$ & $\cdots$ & $\cdots$& $\omega_{c_{r,1}}$\\
$\omega_{c_{1,2}}$   & $\omega_{c_{2,2}}$   & $\omega_{c_{3,2}}$   & $\cdots$ & $\cdots$ & $\cdots$ & $\cdots$ & $\cdots$ & $\cdots$& $\omega_{c_{r,2}}$\\
$\vdots$             & $\vdots$             & $\vdots$             & $\vdots$ & $\vdots$ & $\vdots$ & $\vdots$ & $\vdots$ & $\vdots$& $\vdots$\\
$\omega_{c_{1,r-3}}$ & $\omega_{c_{2,r-3}}$ & $\omega_{c_{3,r-3}}$ & $\cdots$ & $\cdots$ & $\cdots$ & $\cdots$ & $\cdots$ & $\cdots$& $\omega_{c_{r,r-3}}$\\
$\omega_{P_1}$       & $\omega_{P_2}$       & $\omega_{P_3}$       & $\cdots$ & $\cdots$ & $\cdots$ & $\cdots$ & $\cdots$ & $\omega_{P_{r-1}}$& $\omega_{P_r}$\\
$\omega_{J_1}$       & $\omega_{J_2}$       & $\omega_{J_3}$       & $\omega_{J_4}$ & $\omega_{J_5}$& $\cdots$ & $\omega_{J_{r-3}}$& $\omega_{J_{r-2}}$
  & $\omega_{J_{r-1}}$  & $\omega_{J_r}$
\end{tabular}
};

% -------------------------- RECTANGLES --------------------------
% Same coordinates as PSTricks
\draw[cyan, thick] (1.25,-3.6) rectangle (-.05,-.5);
\draw[cyan, thick] (2.85,-3.6) rectangle (3.6,-3.05);
\draw[cyan, thick] (5.2,-3.6) rectangle (6,-3.05);
\draw[cyan, thick] (7,-3.6) rectangle (8.1,-3.05);
\draw[cyan, thick] (10.75,-3.6) rectangle (11.5,-3.05);
\draw[cyan, thick] (9.5,-3.05) rectangle (10.5,-2.55);

\end{tikzpicture}\renewcommand{\figurename}{Table}\caption{}\label{Table1-e}
\end{figure}

For \(B_1\), we place the vertices of the corresponding edge
\[
(\omega_{J_r}, \omega_{c_{1,1}}, \omega_{c_{1,2}}, \ldots,
\omega_{c_{1,r-3}}, \omega_{P_1}, \omega_{J_1})
\]
in the first part of the partition, together with the joint vertices
\[
\omega_{J_3}, \omega_{J_5}, \ldots, \omega_{J_{r-3}},
\]
and the penultimate vertex of the penultimate edge, \(\omega_{P_{r-1}}\).
%Note that the vertices \(\omega_{J_3}, \omega_{J_5}, \ldots, \omega_{J_{r-3}}\)
%are at distance two from each other along the cycle, while the distance
%between \(\omega_{J_{r-3}}\) and \(\omega_{J_r}\) is three.
In Table~\ref{Table1-e}, these elements are highlighted with a blue rectangle:
\[
{B_1}_1
=
\{\omega_{J_r}, \omega_{c_{1,1}}, \omega_{c_{1,2}}, \ldots,
\omega_{c_{1,r-3}}, \omega_{P_1}, \omega_{J_1},
\omega_{J_3}, \ldots, \omega_{J_{r-3}}, \omega_{P_{r-1}}\}.
\]
Now, we omit the elements of the first part of the partition from the table by
removing the first column, which contains the determined joint and penultimate
vertices.  
For each removed joint vertex, we place it together with the penultimate vertex
directly above it, leaving the corresponding penultimate-vertex cell blank.
The cell corresponding to the distinguished penultimate vertex is also left
blank.
This reduction process is illustrated in Table~\ref{Table2-e}.

\begin{figure}[H]
\centering
\begin{tikzpicture}[x=1cm,y=1cm]

% -------------------------- TABLE --------------------------
\node at (6, -2) {
\begin{tabular}{lllllllll}
$\cc{\omega_{c_{2,1}}}$  & $\yy{\omega_{c_{3,1}}}$   & $\cdots$ & $\cdots$ & $\cdots$ & $\cdots$ & $\cdots$ & $\cdots$&$\bb{\omega_{c_{r,1}}}$\\
$\yy{\omega_{c_{2,2}}}$  & $\omega_{c_{3,2}}$ & $\cdots$ & $\cdots$ & $\cdots$ & $\cdots$ & $\cdots$ & $\cdots$ &$\cc{\omega_{c_{r,2}}}$\\
$\vdots$        & $\vdots$        & $\vdots$ & $\bb{\omega_{c_{5,r-4}}}$ & $\vdots$ & $\vdots$ & $\vdots$ & $\vdots$ & $\vdots$\\
$\omega_{c_{2,r-3}}$ & $\omega_{c_{3,r-3}}$ & $\bb{\omega_{c_{4,r-3}}}$ &  $\cc{\omega_{c_{5,r-3}}}$ & $\cdots$ &  $\cdots$ &  $\cdots$& $\omega_{c_{r-1,r-3}}$ & $\omega_{c_{r,r-3}}$\\
$\omega_{P_2}$ & blk      & $\cc{\omega_{P_4}}$ &blk& $\cdots$ &blk& $\omega_{P_{r-2}}$& blk &blk \\
$\bb{\omega_{J_2}}$ & $\cc{\omega_{P_3}}$ & $\yy{\omega_{J_4}}$ 
                & $\omega_{P_5}$ 
                & $\cdots$
                & $\omega_{P_{r-3}}$& $\omega_{J_{r-2}}$
                & $\omega_{J_{r-1}}$
                & $\omega_{P_{r}}$\\
\end{tabular}
};

\end{tikzpicture}\renewcommand{\figurename}{Table}\caption{}\label{Table2-e}
\end{figure}

The remaining parts of the partition are defined in the same way as in Subsection~\ref{rOdd},
using the cyclic diagonals of the reduced table, which now has \((r-1)\) rows and columns.  
As in Subsection~\ref{rOdd}, there is a distinction between diagonals starting in odd and even positions; 
only the penultimate diagonal behaves like an even-starting diagonal.

The first part of the partition contains all the vertices of the first edge, 
together with joint vertices that are at a two-edge distance.  
It does not include any vertices from the penultimate edge, except for its penultimate vertex.
Then, this part satisfies the \(\J\)-minimal compatibility condition. 
Furthermore, the other parts of the partition, except for the penultimate-vertex row, contain at most one element per row. Hence, $B_1$ satisfies the \(\J\)-minimal compatibility condition.

Moreover, as in Subsection~\ref{rOdd}, we may define the remaining parts of the partition 
with respect to each edge of Figure~\ref{fig:nGen}, using the corresponding edges in Table~\ref{Table1-e}.

Similarly, following Subsection~\ref{rOdd}, we define 
\(\mathrm{Dign}_{i,j}\) in \(B_j\) as the set of entries along the diagonal above the penultimate row, 
whose first entry occurs in column~\(i\), namely \(\omega_{c_{i,r-3}}\).

The definition of the partition \(B_r^{\ast}\) is analogous to that in the previous subsection. 
The first part, and each part up to the one preceding the penultimate, consists of the rows 
from the first row up to the row just before the penultimate row of Table~\ref{Table1-e}. 
We place the joint vertices
\(\omega_{J_2}, \omega_{J_5}, \ldots, \omega_{J_{2k+1}}, \ldots, \omega_{J_{r-1}}\),
together with the penultimate vertices \(\omega_{P_1}\) and \(\omega_{P_4}\),
in the first part, and the joint vertices
\(\omega_{J_3}, \omega_{J_6}, \ldots, \omega_{J_{2k}}, \ldots, \omega_{J_r}\),
together with the penultimate vertices \(\omega_{P_2}\) and \(\omega_{P_5}\),
in the second part. All remaining penultimate vertices, together with the joint
vertices \(\omega_{J_1}\) and \(\omega_{J_4}\), are placed in the last part.
In this way, the partition \(B_r^{\ast}\) satisfies the
\(\J\)-minimal compatibility condition. It is also clear that this
definition of \(B_r^{\ast}\) does not work for \(r = 4\).

Therefore, we obtain the desired set  
\[
\mathcal{B}_r=\{\, B_1,\ \ldots,\ B_{r-1},\ B_r^{\ast} \,\}.
\]

Now, we verify that $\mathcal{W}^{\perp} = 0$.

Let
\begin{align*}
A = &\alpha_{J_1}\omega_{J_1}+ \cdots+ \alpha_{J_r}\omega_{J_r}+\\
&\alpha_{P_1}\omega_{P_1}+ \cdots+ \alpha_{P_r}\omega_{P_r}+
 \alpha_{c_{1,1}}\omega_{c_{1,1}}+\cdots+\alpha_{c_{r,r-3}}\omega_{c_{r,r-3}}\in \mathcal{W}^{\perp}.
\end{align*}

%Let \(1 \le j \le r-3\).
%In \(B_j\), we obtain the relation  
%\[
%\alpha_{J_{j+1}} + 1_{\mathrm{Dign}_{j+3,j}}(A) = 0,
%\]
%and in \(B_{j+2}\), we have
%\[
%\alpha_{J_j} + 1_{\mathrm{Dign}_{j+3,j+2}}(A) = 0.
%\]
%Since \(\mathrm{Dign}_{j+3,j} = \mathrm{Dign}_{j+3,j+2}\), it follows that
%\[
%\alpha_{J_{j+1}} = \alpha_{J_j}.
%\]
Exactly as in Subsection~\ref{rOdd}, we obtain relations~\eqref{eqeq01}, \eqref{eqeq02}, and~\eqref{eqeq}.
%\begin{equation}\label{eqeq01e}
%\alpha_{J_1} = \alpha_{J_2} = \cdots = \alpha_{J_{r-2}},
%\end{equation}
%Also, in \(B_1\), we obtain the relation  
%\[
%\alpha_{J_{r-1}} + 1_{\mathrm{Dign}_{2,1}}(A) = 0,
%\]
%and in \(B_{r-1}\), we have
%\[
%\alpha_{J_r} + 1_{\mathrm{Dign}_{2,r-1}}(A) = 0.
%\]
%Again as \(\mathrm{Dign}_{2,1} = \mathrm{Dign}_{2,r-1}\), it follows that
%and
%\begin{equation}\label{eqeq02e}
%\alpha_{J_{r-1}} = \alpha_{J_r}.
%\end{equation}

%Moreover, from the partition \(B_r^{\ast}\), we obtain the relations~(\ref{eqeq}).
%Now, form the partition $B_r^{\ast}$, we obtain the relations
%\begin{equation}\label{eqeq}
%\begin{aligned}
%\alpha_{c_{1,1}}+\alpha_{c_{2,1}}+\cdots+\alpha_{c_{r,1}}&=0,\\
%\alpha_{c_{1,2}}+\alpha_{c_{2,2}}+\cdots+\alpha_{c_{r,2}}&=0,\\
%&\vdots\\
%\alpha_{c_{1,r-3}}+\alpha_{c_{2,r-3}}+\cdots+\alpha_{c_{r,r-3}}&=0.
%\end{aligned}
%\end{equation}
Thus, as in the previous subsection, we again obtain
\begin{equation}\label{jrj1e}
3\alpha_{J_r} + (r-3)\alpha_{J_1} = 0.
\end{equation}

%%%%%%%%%%%%%%%%%%%%%

As in Subsection~\ref{rOdd}, we obtain the relations~\eqref{eqeq2}.
%\begin{equation}\label{eqeq2}
%\begin{aligned}
%\alpha_{P_{r}} + \alpha_{P_{2}} &= \alpha_{J_{r}},\\
%\alpha_{P_{1}} + \alpha_{P_{3}} &= \alpha_{J_{1}},\\
%&\vdots\\
%\alpha_{P_{r-3}} + \alpha_{P_{r-1}} &= \alpha_{J_{r-3}}
%\end{aligned}
%\end{equation}
In particular, these imply that
\[
\alpha_{P_1} = \alpha_{P_5}.
\]

Now, by relations~\eqref{eqeq01}, \eqref{eqeq02}, and from the first and second parts of \(B_r^{\ast}\), we obtain
\[
\alpha_{P_{1}} + \alpha_{P_{4}} + \alpha_{J_{2}} + \alpha_{J_{5}} + \cdots + \alpha_{J_{r-1}}
= \alpha_{P_{1}} + \alpha_{P_{4}} + \Big(\tfrac{r}{2}-2\Big)\alpha_{J_{1}} + \alpha_{J_{r}} = 0
\]
and
\[
\alpha_{P_{2}} + \alpha_{P_{5}} + \alpha_{J_{3}} + \alpha_{J_{6}} + \cdots + \alpha_{J_{r}}
= \alpha_{P_{2}} + \alpha_{P_{5}} + \Big(\tfrac{r}{2}-2\Big)\alpha_{J_{1}} + \alpha_{J_{r}} = 0.
\]
Since \(\alpha_{P_1} = \alpha_{P_{5}}\), it follows that \(\alpha_{P_{2}} = \alpha_{P_{4}}\), and thus, by~\eqref{eqeq2}, we conclude that
\[
\alpha_{P_{2}} = \alpha_{P_{4}} = \cdots = \alpha_{P_{r-2}}.
\]
Hence, we have
\[
\alpha_{P_{4}} + \alpha_{P_{5}}
= -\Big(\tfrac{r}{2}-2\Big)\alpha_{J_{1}} - \alpha_{J_{r}},
\]
and by~\eqref{eqeq2}, it follows that
\begin{equation}\label{pkpk1}
\alpha_{P_{4k+2}} + \alpha_{P_{4k+3}}
= \tfrac{r}{2}\alpha_{J_{1}} + \alpha_{J_{r}},
\end{equation}
and
\begin{equation}\label{pkpk2}
\alpha_{P_{4k}} + \alpha_{P_{4k+1}}
= -\Big(\tfrac{r}{2}-2\Big)\alpha_{J_{1}} - \alpha_{J_{r}},
\end{equation}
whenever the indices \(4k, 4k+1, 4k+2,\) and \(4k+3\) lie between \(4\) and \(r-1\).

Now, from~\eqref{eqeq2}, we have \(\alpha_{P_r} + \alpha_{P_2} = \alpha_{J_r}\) and
\(\alpha_{P_3} + \alpha_{P_5} = \alpha_{J_1}\).
Adding the relations obtained from the last two parts of the partition
\(B_r^{\ast}\), we obtain
\begin{align*}
&\alpha_{P_2} + \alpha_{P_r} + \alpha_{P_3} + \alpha_{P_5}
+ \tfrac{r}{2}\alpha_{J_1} + \alpha_{J_r}
+ \alpha_{P_6} + \cdots + \alpha_{P_{r-1}} \\
&= \Big(\tfrac{r}{2}+1\Big)\alpha_{J_1} + 2\alpha_{J_r}
+ \alpha_{P_6} + \cdots + \alpha_{P_{r-1}} = 0.
\end{align*}

Now, by~\eqref{eqeq2}, \eqref{pkpk1}, and~\eqref{pkpk2}, if \(r = 2k\), we obtain
\[
\Big(\tfrac{3r}{2}-3\Big)\alpha_{J_1} + 3\alpha_{J_r} = 0,
\]
and if \(r = 2k+2\) we obtain
\[
(r-2)\alpha_{J_1} + 2\alpha_{J_r} = 0.
\]
In both cases, together with~\eqref{jrj1e}, it follows that
\(\alpha_{J_1} = \alpha_{J_r} = 0\).

Again, from~\eqref{eqeq2} we have \(\alpha_{P_2} + \alpha_{P_4} = \alpha_{J_1}\).
Since \(\alpha_{J_1}=0\), it follows that
\[
\alpha_{P_2} = \alpha_{P_4} = \cdots = \alpha_{P_{r-2}} = 0.
\]
Thus, by~\eqref{pkpk1} and~\eqref{pkpk2}, we obtain
\[
\alpha_{P_5} = \alpha_{P_7} = \cdots = \alpha_{P_{r-1}} = 0.
\]
Finally, from the relation \(\alpha_{P_r} + \alpha_{P_2} = \alpha_{J_r}\) and the fact that
\(\alpha_{J_r} = 0\), we conclude that \(\alpha_{P_r} = 0\).

%%%%%%%%%%%%%%%%%%%%%

Also, exactly as in the case where \(r\) is odd, the same argument shows that the
coefficients of the remaining vertices are zero. Hence \(A = 0\), and therefore
\(\mathcal{W}^{\perp} = 0\).

%%%%%%%%%%%%%%%%%%%%%%%%%%%%%%%%%%%%%%%%%%%

Let us define the partition \(B_r\) in a manner similar to the other partitions \(B_1, \ldots, B_{r-1}\).
Now, we show that for
\[
\mathcal{B}'_r = \{ B_1, \ldots, B_{r-1}, B_r \},
\]
we have \(\mathcal{W}^{\perp} \neq 0\). Consequently, 
for the monoid \(M\) constructed from \(\mathcal{B}'_r\) and \(\mathcal{I}_r\) as in
\eqref{ConsM}, the augmentation submodule \(\Aug(\mathbb{C}\Omega_r)\) is not
simple where $r$ is even more than 4.

Let
\begin{align*}
A = &\alpha_{J_1}\omega_{J_1}+ \cdots+ \alpha_{J_r}\omega_{J_r}+\\
&\alpha_{P_1}\omega_{P_1}+ \cdots+ \alpha_{P_r}\omega_{P_r}+
 \alpha_{c_{1,1}}\omega_{c_{1,1}}+\cdots+\alpha_{c_{r,r-3}}\omega_{c_{r,r-3}}
\end{align*}
where, for all \(1 \le i \le r\) and all relevant \(j\),
\[
\alpha_{c_{i,j}} =
\begin{cases}
1, & \text{if \(j\) is odd},\\
-1, & \text{if \(j\) is even},
\end{cases}
\qquad
\alpha_{J_i} = -1,
\qquad
\alpha_{P_i} = -\tfrac{1}{2}.
\]
For every \(1 \le k \le r\), each block of the partition \(B_k \in \mathcal{B}'_r\)
contains \(\tfrac{r-2}{2}\) vertices \(\omega_{c_{i,j}}\) with \(j\) odd and
\(\tfrac{r-4}{2}\) vertices \(\omega_{c_{i,j}}\) with \(j\) even.
In addition, each block contains either two penultimate vertices or one joint vertex.
It follows that \[1_{B_{kl}}(A)=0\ \text{for all}\ 1 \le \ell \le r.\]
Hence \(A \in \mathcal{W}^{\perp}\), and the result follows.

%%%%%%%%%%%%%%%%%%%%%%%%%%%%%%%%%%%%%%%%%%%%%%%%%%%%%%%%%%%%%%%%%%%%%%%%%%
%%%%%%%%%%%%%%%%%%%%%%%%%%%%%%%%%%%%%%%%%%%%%%%%%%%%%%%%%%%%%%%%%%%%%%%%%%

We summarize the results obtained in this section for both odd and even integers
greater than 4 in the following theorem.

\begin{thm}\label{thm:main}
Let \( \mathcal{I}_r \), with \( r > 4 \), be the set depicted in Figure~\ref{fig:nGen}.
The monoid \( M \) constructed as in~\eqref{ConsM} from \( \mathcal{B}_r \) together
with \( \mathcal{I}_r \) has a simple augmentation submodule, in contrast to the
monoid constructed from \( \mathcal{B}'_r \) together with \( \mathcal{I}_r \),
whose augmentation submodule is not simple.
\end{thm}

%%%%%%%%%%%%%%%%%%%%%%%%%%%%%%%%%%%%%%%%%%%%%%%%%%%%%%%%%%%%%%%%%%%%%%%%%%
%%%%%%%%%%%%%%%%%%%%%%%%%%%%%%%%%%%%%%%%%%%%%%%%%%%%%%%%%%%%%%%%%%%%%%%%%%

%\section*{Acknowledgments}
%The first author was partially supported by CMUP, which is financed by
%national funds through FCT -- Fundação para a Ciência e a Tecnologia,
%I.P., under the project UIDB/00144/2020. 

\bibliographystyle{abbrv}
\bibliography{ref-Rep-Aug,standard2}

\end{document}